\documentclass[12pt]{article}
\usepackage{fullpage}

\usepackage{url}
\usepackage{amsthm,amsmath,amssymb}
\usepackage[colorlinks=true,citecolor=black,linkcolor=black,urlcolor=blue]{hyperref}
\newcommand{\doi}[1]{\href{http://dx.doi.org/#1}{\texttt{doi:#1}}}
\newcommand{\arxiv}[1]{\href{http://arxiv.org/abs/#1}{\texttt{arXiv:#1}}}

\newcommand{\mathheader}[1]{\texorpdfstring{$\mathbf{#1}$}{#1}}

\newcommand{\fourdots}{{\begin{smallmatrix} 
     \cdot\kern -1.8pt & \cdot \\[-2.7pt]
     \cdot\kern -1.8pt& \cdot \end{smallmatrix}}}

\theoremstyle{plain}
\newtheorem{thm}{Theorem}[section]
\newtheorem{lem}[thm]{Lemma}
\newtheorem{prop}[thm]{Proposition}

\newtheorem{cor}[thm]{Corollary}

\theoremstyle{definition}
\newtheorem{defn}[thm]{Definition}

\theoremstyle{remark}
\newtheorem{rem}[thm]{Remark}

\title{\bf Counting Permutations Modulo Pattern-Replacement Equivalences for Three-Letter Patterns}

\author{William Kuszmaul \\
\small MIT PRIMES \\[-0.8ex]
\small Massachusetts Institute of Technology \\[-0.8ex] 
\small Cambridge, Massachusetts, USA \\
\small\tt william.kuszmaul@gmail.com
}

\begin{document}

\maketitle

\begin{abstract}
We study a family of equivalence relations on $S_n$, the group of permutations on $n$ letters, created in a manner similar to that of the Knuth relation and the forgotten relation. For our purposes, two permutations are in the same equivalence class if one can be reached from the other through a series of pattern-replacements using patterns whose order permutations are in the same part of a predetermined partition of $S_c$. 

When the partition is of $S_3$ and has one nontrivial part and that part is of size greater than two, we provide formulas for the number of classes created in each previously unsolved case. When the partition is of $S_3$ and has two nontrivial parts, each of size two (as do the Knuth and forgotten relations), we enumerate the classes for $13$ of the $14$ unresolved cases. In two of these cases, enumerations arise which are the same as those yielded by the Knuth and forgotten relations. The reasons for this phenomenon are still largely a mystery.
\end{abstract}


\section{Introduction}

In 1970, Donald Knuth discovered the so-called Knuth relation, using it as a tool for studying the RSK (Robinson-Schensted-Knuth) correspondence \cite{Kn}. The RSK correspondence is a bijection between permutations in $S_n$ (for fixed $n$) and pairs of standard Young tableaux of the same shape with $n$ cells. The Knuth relation, also known as the plactic equivalence, was designed to connect two permutations exactly when the first tableau that each one is mapped to under the RSK correspondence is the same for both. However, over time, it would find applications not only to combinatorics, but also to abstract algebra. The Knuth relation is particularly well-known for its aid in the proof of the Littlewood-Richardson rule, an identity which can be interpreted as a multiplication rule for Schur polynomials (among many other things). This proof was historically one of the first proofs to be found for the rule.

In 2008, Novelli and Schilling brought to light the so-called forgotten relation \cite{NS}, originally discovered by Lascoux and Sch\"{u}tzenberger in 1981, but heretofore ignored in most existing literature (whence the name). The forgotten relation and the Knuth relation share a common structure: Each of them is an equivalence relation on $S_n$ (or, more generally, on some set of words) which connects two permutations obtained from each other by a rearrangement of three consecutive letters according to certain rules. The forgotten relation also shares a common application with its fore-father. When either of the forgotten or Knuth relations is forced on the free associative algebra, we get a structure in which the elementary symmetric functions (appropriately defined) commute. The similar structure and uses of these two relations inspired two research groups to systematically analyze relations in the same family in \cite{LPRW10} and \cite{PRW11}.

There are several ways to set the rules that allow rearranging letters in a permutation. The ones we are going to deal with allow the rearranging of more than three consecutive letters, but the number of letters rearranged is fixed, and the rules which determine when rearrangements are allowed to depend only on the relative order of the letters we are rearranging and on the order in which we want to rearrange them. This is what \cite{LPRW10} and \cite{PRW11} call the ''only indices adjacent'' case (or the $P^{\mid\ \mid}$ case). We formalize the rules into the notion of a replacement partition (see the definitions below). Both \cite{LPRW10} and \cite{PRW11} studied only replacement partitions with one non-trivial part using patterns of size three. They considered two problems: First, how many equivalence classes does a given such relation subdivide $S_n$ into? Second, how big is the class containing the identity? Along with solving the second problem, they were able to solve all but seven cases of the first, as long as the nontrivial part of the replacement partition is required to be of size greater than two. We extend their work by solving the remaining seven cases (Section \ref{Secsingle}). We furthermore study a family of relations more closely related to the Knuth and forgotten relations, by considering replacement partitions of $S_3$ having two non-trivial parts, each of size two. We enumerate the classes in all but one of the unresolved cases (Section \ref{Secdouble}). Finally, in Section \ref{Secgen}, we study several general results concerning things such as a connection between pattern-avoidance and the enumeration of equivalence classes. These results are sometimes cited in proofs from previous sections. We conclude with comments on methodology and future work (Section \ref{conclusion}).

Before we present our results, we establish some conventions. Most of these formalize conventions made in \cite{LPRW10}.

\begin{defn} A \emph{word} is a finite, possibly empty sequence of
positive integers.  The elements of a word are called \emph{letters}.
We denote the $i$-th letter of word $w$ as $w_i$, with $w_1$ as the
first letter.
\end{defn}
\begin{defn} We regard any permutation of $\left\{1,2,\ldots ,n\right\}$ as a word (by writing it in one-line notation). Conversely, given a word $w$ with no two letters equal, we define the \emph{order permutation} of $w$ as the unique permutation $\pi\in S_n$ (where $n$ is the length of $w$) such that for any $i$ and $j$, we have $\pi_i<\pi_j$ if and only if $w_i<w_j$. This permutation $\pi$ is also known as the \emph{standardization} of $w$. 
\end{defn}
For example, the order permutation of $425$ is $213$. The notion of an order permutation is more generally defined for any words (possibly with equal letters), which requires a subtler definition, but we are never going to need it in this generality.

\begin{defn}
If a word $w$ has order permutation $w'$, then we may simply say that $w$ \emph{forms} the permutation $w'$.
\end{defn}

\begin{defn}
A \emph{factor} of a permutation is a subsequence of adjacent letters in the permutation.
\end{defn}

Note that we may treat a factor either as a word or as its order permutation.

\begin{defn}
The \emph{position parity} of a letter in a permutation is the parity of the position of the letter.
\end{defn}

For example, the first letter in a permutation has odd position parity.

Throughout the paper, we may sometimes use notation like ``$f(n<3)=n+2$'' which is just shorthand for ``when $n<3$, $f(n)=n+2$''.

\begin{defn}
If $a$ and $b$ are two words, then the \emph{concatenation} of $a$ with $b$ is defined as the word obtained by attaching $b$ to the end of $a$. It is denoted by $ab$ or $a.b$. (\emph{Note:} If $\pi$ is a permutation of size $n$, then $\pi n$ denotes the concatenation $\pi.n$, not the number $\pi\left(n\right)$.)
\end{defn}
For example, the concatenation of $21$ with $72$ is $2172$.
\begin{defn} Let $\pi\in S_n$. Let $w$ be a word with no two letters equal. If we can write $w$ in the form $aub$ for some three (possibly empty) words $a$, $u$ and $b$ such that $u$ has order permutation $\pi$, then we say that $u$ is \emph{a $\pi$ in $w$} (or \emph{a $\pi$ pattern in $w$}). We say that $w$ \emph{avoids $\pi$} if there is no $\pi$ in $w$.
\end{defn} 
For example, $574$ is a $231$ in $2657431$ because the order permutation of $574$ is $231$. However, $3124$ avoids $231$.

We now will define the equivalence relations that we are going to study.
\begin{defn} Let $k\in \mathbb{N}$. A \emph{replacement partition} of $S_k$ is a set partition of the symmetric group $S_k$.
\end{defn}
\begin{defn} Given a replacement partition $K$ of $S_k$ and a positive integer $n$, the \emph{$K$-equivalence on $S_n$} is defined as the equivalence relation on $S_n$ generated by the following requirement:
If $\phi\in S_n$ and $\psi\in S_n$ are such that $\phi = aub$ and $\psi = avb$ for some words $a$, $b$, $u$ and $v$, where $u$ and $v$ have length $k$, and the order permutation of $u$ lies in the same part of $K$ as the order permutation of $v$, then we say that $\phi$ is \emph{equivalent} to $\psi$.

Moreover, in this case, we say that $\psi$ results from $\phi$ by a \emph{$K$-transformation}, or more precisely, $\psi$ results from $\phi$ by a \emph{transformation $p\to q$}, where $p$ is the order permutation of $u$ and $q$ is the order permutation of $v$.

When $K$ is clear from the context, we abbreviate ``$K$-equivalence'' as ``equivalence'', and ``$K$-transformation'' as ``transformation''.

We also write $\phi \equiv \psi$ for ``$\phi$ is equivalent to $\psi$''.
\end{defn}

Note that if two permutations are equivalent, we may say that they are \emph{reachable} from each other. This tends to be used when we are providing the manner through which we can reach one from the other.

What we call ``$K$-equivalence'' is denoted as ``$K^{\mid\ \mid}$-equivalence'' in \cite{LPRW10}.\\
{\bf Example: } Let $n = 5$, $k = 3$, and $K = \left\{\left\{123,321\right\}, \left\{132,231\right\}, \left\{213\right\}, \left\{312\right\}\right\}$. We will later abbreviate this by $K = \left\{123,321\right\}\left\{132,231\right\}$, leaving out the outer brackets and the one-element parts of the partition. Then, the permutation $15324\in S_5$ is $K$-equivalent to $12354\in S_5$ (because $15324 = aub$ and $12354 = avb$ with $a = 1$, $u = 532$, $v = 235$ and $b = 4$, and the order permutation of $532$ lies in the same part of $K$ as the order permutation of $235$). More precisely, $12354$ results from $15324$ by a transformation $321\to 123$. Similarly, $12354$ is equivalent to $12453$ (here, $a = 12$, and $b$ is the empty word), and $12453$ results from $12354$ by a transformation $132\to 231$. Combining these, we see that $15324$ is equivalent to $12453$, although $12453$ does not directly result from $15324$ by any transformation.
\begin{defn} Given a replacement partition $K$ and a positive integer $n$, the $K$-equivalence on $S_n$ partitions $S_n$ into equivalence classes. We will briefly refer to these equivalence classes as \emph{classes}. A class is called \emph{trivial} if it consists of one element only.
\end{defn}
\begin{defn} Let $K$ be a replacement partition. If $w$ is a word with no two letters equal, then a \emph{hit} (or more precisely, a \emph{$K$-hit}) in $w$ is a word $u$ such that $w = aub$ for some words $a$ and $b$, and such that the order permutation of $u$ lies in a nontrivial part of $K$. A word with no two letters equal is said to \emph{avoid} $K$ if it contains no hit, i.e., if it avoids every permutation in every nontrivial part of $K$. Otherwise it is said to be \emph{non-avoiding}, or, equivalently, \emph{a non-avoider} (with respect to $K$).
\end{defn}

Observe that each permutation that avoids $K$ forms a trivial class (with respect to the $K$-equivalence), whereas non-avoiders lie in nontrivial classes. 

We will occasionally refer to the elements of $K$ as \emph{patterns}, hoping that no confusion with the notation of ``$\pi$ pattern in $w$'' can arise.

If two permutations are connected by a transformation, then we say the letters in the hit used in the transformation are \emph{involved} in the transformation, even if they are static in the transformation.

Figure~\ref{answers} shows the number of classes created in $S_n$ by the replacement partitions considered in this paper. Each of these results are proven in the following two sections.

\begin{figure}
\begin{center}
\[
\begin{array}{| c | l |}
\hline
\!\!\mbox{replacement partition}\!\!            & \mbox{number of classes in $S_n$} \\ \hline
\{213, 231, 132\}        & 2^{n-2}+2n-4                          \\ \hline
\{123, 132, 231\}        & 2^{n-1}                               \\ \hline
\{123, 132, 321\}        & (n-1)!!+(n-2)!!+n-2                  \\ \hline
\{123, 132, 312\}        & f(n\ge 5)=f(n-1)+(n-2) \cdot f(n-2)+1  \\ \hline
\{123, 132, 213, 231\}   & n                                    \\ \hline  
\{123, 132, 231, 321\}   & 2 \mbox{ for $n>3$}                  \\ \hline
\{213,132,231,312\}      & 3                                    \\ \hline
\{123, 132\} \{312, 321\}& 2^{n-1}                               \\ \hline
\{123, 132\} \{213, 231\}& 2^{n-1}                               \\ \hline
\{123, 231\} \{132, 321\}& 2^{n-1}                               \\ \hline
\{132, 312\} \{321, 213\}& (n^2+n)/2-2                          \\ \hline
\{123, 231\} \{213, 132\}& n^2-3n+4                             \\ \hline
\{123, 321\} \{213, 231\}& 3 \mbox{ for $n>5$}                  \\ \hline
\{123, 132\} \{231, 312\}& 3 \cdot 2^{n-3} + n - 2 \mbox{ for $n>5$} \\ \hline
\{123, 132\} \{213, 321\}& \mbox{Sum of the first $n-1$ Motzkin numbers} \\ \hline
\{123, 132\} \{213, 312\}& f(n \ge 3)=f(n-1)+(n-1) \cdot f(n-2) \\ \hline
\{123, 321\} \{132, 213\}& \binom{n}{\lfloor n/2 \rfloor}+\binom{n-2}{\lfloor (n-2)/2 \rfloor}+3 \mbox{ for $n>4$} \\ \hline
\{123, 231\} \{321, 213\}&

f(n>5)=
\begin{cases}
3n, & \text{ if } n \text{ is even} \\
3n-1, & \text{ if } n \text{ is odd}
\end{cases}
\\ \hline
\{123, 321\} \{132, 231\}&  \sum\limits_{x=1}^{l}{ x! \cdot \binom{n-x-1}{h-1}} + \sum\limits_{x=1}^{h}{ x! \cdot \binom{n-x-1}{l-1}}, \\
                         &  \mbox{where $l = \lfloor n/2 \rfloor$, $h= \lceil n/2 \rceil$.}                                          \\ \hline
\{123, 231\} \{213, 312\}& g(n,k)= \begin{cases}
   1 & \begin{array}{l} \text{if } n=1 \text{ or } \\ n-2k+1=0, \hskip-7pt \end{array} \\
   \sum\limits_{j=1}^{\lfloor (n+1)/2 \rfloor}{g(n-1, j)} & \text{ if } k=1,  \\
   \sum\limits_{x=k-1}^{n-k}{\sum\limits_{j=1}^{n-k-x}{\binom{x-1}{k-2} \cdot g(n-2k+1, j)}}  & \text{ otherwise.}
  \end{cases}
 \\
& f(n)=\sum\limits_{k=1}^{\lfloor n/2+1 \rfloor}{g(n+1, k)}+n-2 \\ \hline

\end{array}
\]
\end{center}
\caption{The number of classes created in $S_n$ by various replacement partitions of $S_3$. Unless otherwise specified, $n\ge3$. In the table, $f(n)$ equals the number of classes created and is used for recursive formulas.} 
\label{answers}
\end{figure}

\section{Single Replacements}\label{Secsingle}
In this section, we consider the number of classes created in $S_n$ by replacement partitions of $S_3$ with exactly one nontrivial part. This problem was previously addressed in some cases by \cite{LPRW10} and \cite{PRW11}. We provide formulas for the number of classes in all cases where the nontrivial part of the replacement partition is of size greater than two, apart from those that have been previously solved. 

When considering replacement partitions, it is important to note that some relations are equivalent to others. In fact, let $w$ be the permutation $k(k-1)\ldots 1 \in S_k$ (where $k$ is the size of the permutations in the replacement partition, and needs not be $3$ in this argument). Without changing the structure of the classes, one can either replace each pattern $\pi$ in the replacement partition by $\pi\circ w = \pi_k\pi_{k-1} \ldots \pi_1$, or replace each pattern $\pi$ in the replacement partition by $w\circ \pi = (k-\pi_1+1)(k-\pi_2+1) \ldots (k-\pi_k+1)$. By these two operations (which generate a Klein four-group if $k\geq 3$), the set of all replacement partitions of $S_k$ is subdivided into orbits (mostly of size $4$, but occasionally smaller), and if we can count the number of classes which are generated in $S_n$ by one partition in each orbit, we automatically obtain the same number for all the other partitions in all of the orbits. Noting this, we will only provide results for one replacement partition per orbit.

The following subsections are concerned with one replacement partition each. In each subsection, the replacement partition $K$ is to be understood to be the partition mentioned in the title of the subsection.

\subsection{\mathheader{\{123, 132, 321\}}-Equivalence}

Note that in this section, we use several results and definitions which we state in Subsection \ref{secstooge}.

Let $k\in\left\{1,2,\ldots \right\}$. We say that a word $w$ satisfies \emph{property $A_k$} if there exists a strictly increasing sequence $\left(i_1,i_2,\ldots,i_k\right)$ of integers \textbf{of equal parity} satisfying\[
w_{i_1} = k,\ \ \ w_{i_2} = k-1,\ \ \ w_{i_3} = k-2,\ \ \ \ldots,\ \ \ w_{i_k} = 1\ \ \ 
\]
(i.e., satisfying $w_{i_j} = k-j+1$ for every $j\in\left\{1,2,\ldots,k\right\}$).

\begin{lem}\label{lemAkinv}
Property $A_k$ is invariant under the $\{123, 132, 321\}$-equivalence. The position parity of the letter $1$ is also invariant.
\end{lem}
\begin{proof}
The second statement is trivial to see.

Hence, we need to show that if a permutation in $S_n$ satisfies property $A_k$ for some $k$, then every permutation equivalent to it also satisfies $A_k$ for the same $k$. Indeed, whenever we perform one of the transformations $123\leftrightarrow 132$, $132\leftrightarrow 321$, and $321\leftrightarrow 123$, the position parity of a letter $i$ does not change as long as every letter smaller than $i$ lies to the right of $i$ and has the same position parity as $i$. But as long as property $A_k$ holds, this is guaranteed for any of the letters $k$, $k-1$, $\ldots$, $1$. So, the position parity of letters $\leq k$ cannot change under any of these transformations as long as $A_k$ holds. Moreover, as long as $A_k$ is satisfied, for every $j\in\left\{1,2,\ldots,k-1\right\}$, the letter $j$ is at least two positions to the right of the letter $j+1$ in our permutation (since $i_1$, $i_2$, \ldots, $i_k$ are of equal parity and increase), and thus it cannot ''jump'' over $j+1$ in a single application of one of the transformations $123\leftrightarrow 132$, $132\leftrightarrow 321$, $321\leftrightarrow 123$ \footnote{In fact, if a letter is at least two places to the right of another letter, then the only way it can jump over that other letter in one single application of one of these transformations is when these two letters are the last and the first letter of a hit. But the letters $j$ and $j+1$ cannot be the last and the first letter of any hit, respectively.}, i.e., it stays to the right of $j+1$. Thus the order of the letters $k$, $k-1$, $\ldots$, $1$ in our permutation does not change. Hence, after we apply a transformation, the permutation still satisfies property $A_k$ for the same $k$.
\end{proof}

Before continuing, we will first define a set $B_n,C_n \subset S_n$.

\begin{defn}
The permutation $b_n\in S_n$ is $(1,2,\ldots,n-2) \circ (1,2,\ldots, n-4) \circ (1,2, \ldots, n-6) \circ \cdots$, where the product ends with $(1)$ if $n$ is odd and with $(1,2)$ if $n$ is even.
\end{defn}

For example, $b_3=123$, $b_4=2134$, $b_5=23145$, and $b_6=324156$.

\begin{defn}
We say that $B_2=\{ \}$. Then, $B_{n>2}$ is the set containing the following:
\begin{enumerate}
\item $wn$ for each $w \in B_{n-1}$;
\item $b_n$. 
\end{enumerate}
\end{defn}

For example, $B_4=\{1234, 2134\}$, $B_5=\{12345, 21345, 23145\}$, and \\ $B_6=\{123456, 213456, 231456, 324156\}$. Note that $B_n$ has $n-2$ elements.

\begin{defn}
We define $C_{n>2}=B_n \setminus \{b_n\}$ along with the permutation \[
(1,2,\ldots,n) \circ (1,2,\ldots, n-4) \circ (1,2, \ldots, n-6) \circ \cdots\]
 in $S_n$, where the product ends with $(1)$ if $n$ is odd and with $(1,2)$ if $n$ is even.
\end{defn}

For example, $C_6=\{123456, 213456, 231456, 324561\}$.

\begin{lem}\label{lem3}
Let $w$, $w'$ be non-avoiding permutations in $S_n$. If $w$ and $w'$ have property $A_k$ for the same $k$ and each has $1$ in the same position parity, then $w \equiv w'$ under the $\{123, 132, 321\}$-equivalence.
\end{lem}
\begin{proof}
We will prove this inductively, with an inductive base case of $n\le 5$. Let $n>5$. Noting the inductive hypothesis, Lemma \ref{lemAkinv} implies that $L_{n-1}=B_{n-1}$ and $R_{n-1}=C_{n-1}$ (as $L_{n-1}$ and $R_{n-1}$ are defined in Subsection \ref{secstooge}). Since $n>5$, by Proposition \ref{propreachmiddled}, every non-avoiding permutation in $S_n$ is equivalent to some middled permutation. By Theorem \ref{thmstoogesort}, this implies that every permutation in $S_n$ is equivalent to some element of $C_n$. Lemma \ref{lemAkinv} leads us to conclude that thus, the $k$ for which $A_k$ holds along with the position parity of the letter $1$ form a complete invariant for non-avoiding permutations in $S_n$ under the $\{123, 132, 321\}$-equivalence.
\end{proof}

\begin{rem}
Noting Lemmas \ref{lemAkinv} and \ref{lem3}, we see that every non-trivial class in $S_n$ under the equivalence has exactly one representative in $B_n$. When given a non-avoider $w \in S_n$, the $B_n$-representative of its class can be easily found arithmetically by checking which of the relations $A_k$ are satisfied for $w$ and checking the position parity of the letter $1$ in $w$.
\end{rem}

\begin{prop}\label{prop1}
Let $n\ge 5$. In $S_n$, there are $n-2$ non-trivial classes under the $\{123, 132, 321\}$-equivalence.
\end{prop}
\begin{proof}
This follows from Lemma \ref{lemAkinv} and Lemma \ref{lem3}.
\end{proof}

\begin{prop}\label{firstprop}
Let $n\ge 5$. There are $(n-1)!!+(n-2)!!+n-2$ classes in $S_n$ under the $\{123, 132, 321\}$-equivalence.
\end{prop}
\begin{proof}
By Proposition \ref{prop1}, there are $n-2$ non-trivial classes. By Theorem 3 of \cite{Ki}, there are $(n-1)!!+(n-2)!!$ trivial classes. 
\end{proof}

\begin{rem}
Proposition \ref{firstprop} is of particular interest because \cite{LPRW10} collected computational data in an attempt to enumerate the $\{123, 132, 321\}$-equivalence, but was unable to find a formula for the number of classes in $S_n$. Proposition \ref{firstprop} provides this formula.
\end{rem}

\subsection{\mathheader{\{123, 132, 231\}}-Equivalence}

\begin{defn}
Let $w$ be a word with no repeating letters. A \emph{left to right minimum} of $w$ is a letter which is smaller than each letter to its left in $w$.
\end{defn}

\begin{defn}
Let $w$ be a word with no repeating letters. If $x$ is a letter of $w$, then the \emph{$x$-min} of $w$ is the first letter to $x$'s right that is a left to right minimum. The $x$-min needs not exist for all $x$.
\end{defn}

\begin{defn}
We call a letter $x$ of a word $w$ (with no repeating letters) \emph{odd-tailed} if $x$ is a left to right minimum and the $x$-min exists and is of different position parity than $x$.
\end{defn}

\begin{lem} \label{tailedlem}
Let $x$ be an odd-tailed letter of a permutation $z \in S_n$. Then, $x$ is odd-tailed in any permutation equivalent to $z$ under the $\{123, 132, 231\}$-equivalence.
\end{lem}
\begin{proof}
Let $x$ be an odd-tailed letter of the permutation $z$ in $S_n$. Let $y$ be the $x$-min. Notice that by their definitions, $x$ is less than everything to its left and $y$ is also less than everything to its left (including $x$). As a consequence, $x$ and $y$ cannot be adjacent letters in any $123$ pattern, $132$ pattern, or $231$ pattern, and because $x$ and $y$ have different position parity, they can not be in the same hit.

In a hit of the form $123$, the only possible position for $x$ is the first (since $x$ is a left to right minimum). The same holds for a hit of the form $132$. In a hit of the form $231$, the first and the third positions are possible candidates for containing left to right minima, but $x$ cannot be in the first (because then, the $x$-min would be in the third place, but that contradicts the fact that the $x$-min is of different position parity than $x$). Hence, if $x$ lies in a hit, then $x$ is acting as the $1$ in the hit. So, after rearrangement, $x$ is still a left to right minimum and maintains its position parity. 

What remains to be shown is that the position parity of the $x$-min is unchanged after a transformation. This is trivial if $y$ is not in the hit being transformed. If $y$ is in the hit, then there are three subcases. If $y$ is the final letter in the hit then the hit must be of the form $231$, and after any rearrangement, $y$ is still the $x$-min and has the same position parity. We know that $y$ can not be the second letter in a hit because a left to right minimum can never be the second letter in a hit. The final possibility is that $y$ is the first letter in the hit. Then, any rearrangement of the hit is either $123 \leftrightarrow 132$, in which case $y$ remains the $x$-min with the same position, or is one of the rearrangements $231 \leftrightarrow 123$ and $231 \leftrightarrow 132$. In the latter scenario, the new $x$-min is either the first or the final letter in the hit after the rearrangement and thus has the same position parity as $y$ does in $z$. 
\end{proof}

\begin{defn}
A permutation is a \emph{V-permutation} if its letters decrease until the letter $1$ and then increase until the end of the permutation.
\end{defn}
 For example, $531246$ is a V-permutation in $S_6$. Note that there are $2^{n-1}$ V-permuta\-tions in $S_n$ because for each letter other than $1$, we can only choose if it is in the decreasing part or the increasing part of the permutation.

\begin{prop}
Let $n\geq 3$. In $S_n$, $2^{n-1}$ classes are created under the $\{123, 132, 231\}$-equivalence.
\end{prop}
\begin{proof}
We will first prove that permutations in $S_n$ are each equivalent to some V-permut\-ation by inducting on $n$ with a trivial base case of $n=3$. Take $w \in S_n$. If $n$ is in the first position, leave it there. Otherwise, slide $n$ to the final position through repeated applications of $132 \rightarrow 123$ and $231 \rightarrow 123$. Applying the inductive hypothesis to the other $n-1$ letters, we reach a V-permutation. Since there are $2^{n-1}$ V-permutations, there are at most $2^{n-1}$ classes. 

Note that the set of odd-tailed letters in a V-permutation is exactly the set of letters to the left of the letter $1$. (They are odd-tailed because for a letter $x$ to the left of $1$, $x$ is a left to right minimum and the $x$-min is the letter directly to $x$'s right and thus is of different position parity than $x$.) This set is distinct for each V-permutation, so by Lemma \ref{tailedlem}, none of the V-permutations are equivalent to each other. Hence, there are $2^{n-1}$ classes, each determined by the set of odd-tailed letters in each of its permutations. 
\end{proof}

\begin{rem}\label{remshortproof}
For $n>5$, we can actually prove the above proposition without using any knowledge about invariants. In fact, the proposition simply falls from Theorem \ref{thmavoiders} (stated in Subsection \ref{Secconnect}). However, this proof provides far less insight into the workings of the equivalence than does the above proof.
\end{rem}

\subsection{\mathheader{ \{123, 132, 312\} }-Equivalence} 

In this subsection, let $f(n)=$ the number of classes created in $S_n$ under the $\{123, 132, 312\} $-equivalence.

\begin{lem}\label{lem4}
 Every permutation not ending with $1$ is equivalent to a permutation of the form $\ldots 1j$ for some $j$ under the $\{123, 132, 312\} $-equivalence.
\end{lem}
\begin{proof} Such a permutation is reached through repeated applications of $123 \rightarrow 312$ and $132 \rightarrow 312$ using the actual letter $1$.
\end{proof}

\begin{lem}\label{lem5} Permutations of the form $\ldots12$ are equivalent under the $\{123, 132, 312\} $-equivalence.
\end{lem} 
\begin{proof}
Let $w=345 \ldots 12$ where the letters not shown are in increasing order. We will show that $w$ is equivalent to an arbitrary permutation $x$ of the form $\ldots 12$. The base cases of $S_3$ and $S_4$ for the following induction are easy to show computationally. If the lemma holds for $S_{n-1}$, then from $x$, we can rearrange the right-most $n-1$ letters however we want, while not moving the final two. We arrange them to be in increasing order, arriving at $x'$. Then, we rearrange the final three letters of $x'$ with $312\rightarrow 123$, arriving at $x''$ where $1$ and $2$ are the actual letters $1$ and $2$ and $3$ is \textbf{not} the actual letter $3$ (because $n>4$). Performing the inductive hypothesis on the left-most $n-1$ letters of $x''$, the actual letter $3$ is placed in the left-most position and we get $w'''$ (requires $n\ge 5$). Rearranging the final $3$ letters of $w'''$, $123 \rightarrow 312$, we reach $y$. Finally, applying the inductive hypothesis to the right-most $n-1$ letters of $y$, we reach $w$.
\end{proof}

\begin{cor}\label{cor5a} Permutations of the form $1\ldots$ are equivalent to each other under the $\{123, 132, 312\} $-equivalence.
\end{cor}
\begin{proof}
By Lemma \ref{lem4}, every permutation of the form $1\ldots$ is equivalent to a permutation of the form $\ldots 1j$ for some $j$. This $j$ must be $2$ (since the smallest letter to the right of $1$ never changes, and that smallest letter is $2$ in the original permutation). Thus every permutation of the form $1\ldots$ is equivalent to a permutation of the form $\ldots 12$. Since any two permutations of the latter kind are equivalent (by Lemma \ref{lem5}), so must be any two permutations of the former kind.
\end{proof}

\begin{lem}\label{lem6} No two permutations of the form $\ldots 1x$ and $\ldots 1y$ where $x \neq y$ are equivalent under the $\{123, 132, 312\} $-equivalence.
\end{lem}
\begin{proof} This falls easily from the fact that the smallest letter to the right of $1$ never changes under the equivalence.
\end{proof}

\begin{lem}\label{lem6.5}
Let $a \neq 2$ be a letter. Let $x$ and $y$ be equivalent permutations ending with $1a$. Let $x'$ and $y'$ be the permutations made from the first $n-2$ letters of $x$ and $y$ respectively. Then, $x'$ is equivalent to $y'$ under the $\{123, 132, 312\} $-equivalence.
\end{lem}
\begin{proof}
Let $x$, $y$, $x'$, and $y'$ be as described. Consider a sequence of permutations that starts with $x$ and ends with $y$, with each pair of consecutive permutations in the sequence connected by a single transformation. In such a sequence, going from one permutation to the next, $a$ can never pass $1$ and $1$ can never pass $2$ because the smallest letter to the right of $1$ can not change. Note that since $1$ and $2$ can never be in the same hit because in the hits considered, $1$ is always to the right of $2$. Thus neither $1$ nor $a$ can be in the same hit as any letter that is to the left of $2$. Consider two consecutive permutations in the sequence, $t$ and $w$. Let $t'$ and $w'$ be $t$ and $w$ respectively with $1$ and $a$ slid to the two right-most positions as $1a$. If the transformation between $t$ and $w$ involves either $1$ or $a$, then it involves letters only to the right of the letter $2$. However, these letters (excluding $1$ and $a$) can be rearranged freely in $t'$ and $w'$ by Corollary \ref{cor5a}. Thus $t'$ and $w'$ are reachable from each other using transformations that do not involve $1$ or $a$. This, of course, also holds if the transformation between $t$ and $w$ involves neither $1$ nor $a$ (because then, we can simply apply the same transformation to get $w'$ from $t'$). Thus there exists a sequence of permutations in $S_{n-2}$ starting with $x$ and ending with $y$, with each pair of consecutive permutations in the sequence connected by transformations which do not involve the final two letters $1a$. Hence, $x' \equiv y'$.
\end{proof}

\begin{lem}\label{lem7} Let $a\neq 2$. Permutations of the form $\ldots1a$ break into $f(n-2)$ classes under the $\{123, 132, 312\} $-equivalence for $n\ge 5$.
\end{lem}
\begin{proof}
They clearly fall into at most $f(n-2)$ classes. By Lemma \ref{lem6.5}, there are at least $f(n-2)$ classes.
\end{proof}

\begin{thm} $f(n)=f(n-1)+(n-2) \cdot f(n-2)+1$ when $n \ge 5$. As base cases, $f(3)=4$, and $f(4)=9$.
\end{thm}
\begin{proof} Assume $n \ge 5$ (The lower cases are computed manually). The permutations ending with $1$ clearly fall into $f(n-1)$ classes because the final letter is immobile. By Lemma \ref{lem4}, every permutation not ending with $1$ is reachable from $\ldots 1j$ for some $j$ and by Lemma \ref{lem6}, this $j$ is unique. Permutations reachable from $\ldots 12$ fall into $1$ class by Lemma \ref{lem5}. Permutations reachable from $\ldots 1j$ where $j \neq 2$ fall into $f(n-2)$ classes
for a given $j$ by Lemma \ref{lem7}. There are $n-2$ possibilities for such a $j$. Thus $f(n)=f(n-1)+(n-2) \cdot f(n-2)+1$ when $n \ge 5$, $f(3)=4$, and $f(4)=9$.
\end{proof}

\begin{rem} The referee for this paper made the following interesting observation. Using the generating function $f(x)=\sum_{n\geq4}f(n)x^{n-1}/(n-1)!$, we can obtain that $$d/dx (f(x)-2x^2-1.5x^3)=f(x)-2x^2+x f(x)+(e^x-1-x-x^2/2),$$
which implies that $$f \left( x \right) = \left( -{{\rm
e}^{-x-1/2\,{x}^{2}}}-2\,{{\rm e}^{
-x-1/2\,{x}^{2}}}x+1/2\,\sqrt {\pi }\sqrt {2} {{\rm
erf}\left(1/2\,\sqrt {2}x\right)}+1 \right) {{\rm e}^{1/2\,x
 \left( 2+x \right) }}.$$
\end{rem}

\subsection{\mathheader{ \{213, 132, 231\} }-Equivalence}

\begin{defn}
In this subsection, let $f(m)=$ the number of non-trivial classes created in $S_m$ under the $\{213, 132, 231\}$-equivalence. Let $g(m)=$ the number of trivial classes created under the $\{213, 132, 231\}$-equivalence in $S_m$.
\end{defn}

\begin{defn}
A permutation in $S_n$ is \emph{reductive} if it satisfies the following three conditions.
\begin{itemize}
\item it does not start with $n-2$; 
\item it ends with $(n-1)n$;
\item its left-most $n-1$ letters are non-avoiding.
\end{itemize}
\end{defn}

\begin{defn}
A permutation in $S_n$ is \emph{decent} if it satisfies the following two conditions.
\begin{itemize}
\item it starts with $n-2$; 
\item its right-most $n-1$ letters are non-avoiding, do not begin with $n-1$, and do end with $n$.
\end{itemize}
\end{defn}

\begin{lem}\label{lemob1}
Let $n>3$. All non-avoiding permutations not beginning with $n-1$ but ending with $n$ are equivalent in $S_n$ under the $\{213, 132, 231\}$-equivalence.
\end{lem}
\begin{proof}
We will prove this inductively using the easily checked base case of $S_4$. Let $n\geq 5$ and assume that the result holds for $S_{n-1}$. Note that all reductive permutations are reachable from each other by the inductive hypothesis applied to the left-most $n-1$ letters.

Let $w\in S_n$ be a non-avoiding permutation that is not beginning with $n-1$ and that ends with $n$. We will prove that $w$ is equivalent to a reductive permutation. Since reductive permutations are all equivalent, this will complete the proof. Applying repeatedly either $132 \rightarrow 213$ or $231 \rightarrow 213$ to $w$, we place $n-1$ in the second to final position (while keeping $n$ in the final position). If the resulting permutation does not begin with $n-2$, then it is reductive (because its left-most $n-1$ letters are clearly non-avoiding, given how we constructed it) and we are done. In the remaining case, the resulting permutation $y$ begins with $n-2$.

Note that $y$ begins with $n-2$ and ends with $(n-1)n$. Since $y$ begins with $n-2$ and is of size $\geq 5$, the first three letters of $y$ form a $312$ or $321$ and thus do not form a hit. Since $y$ is non-avoiding and cannot begin with a hit, the right-most $n-1$ letters of $y$ must be non-avoiding. Thus $y$ is decent. Any two decent permutations are equivalent to each other by the inductive hypothesis applied to the right-most $n-1$ letters. Thus we only need to show that some decent permutation is equivalent to a reductive permutation.

Consider the permutation $w=(n-2)1(n-1)2n \ldots$. Note that the letters not shown are those between $2$ and $n-2$, and are in increasing order from left to right. Sliding $n$ to the final position through repeated applications of $132 \rightarrow 213$, and then sliding $(n-1)$ to the second to final position in the same manner, we reach $w'=(n-2) \ldots 21(n-1)n$. One can easily check that $w'$ is decent. From $w$, we can obtain $w''=21(n-1)n(n-2) \ldots$ in the following way.
$$(n-2)1(n-1)2n \ldots \rightarrow$$
$$(n-2)1(n-1)n2 \ldots \rightarrow$$
$$1(n-1)(n-2)n2 \ldots \rightarrow$$
$$1(n-1)2n(n-2) \ldots \rightarrow$$
$$21(n-1)n(n-2) \ldots \phantom{\rightarrow}$$

From $w''$, we can first slide $n$ to the final position through repeated applications of $231$ or $132$ $\rightarrow 213$, and then slide $n-1$ to the second to final position in the same manner to get $w'''$. Note that $w'''$ begins with the letter $2$, ends with the letters $(n-1)n$, and has a hit in the penultimate three letters. Hence, $w'''$ is a reductive permutation and we are done.
\end{proof}

\begin{lem}\label{lemob2}
Let $n>3$. All non-avoiding permutations of the form $(n-1) j \ldots n$ with $j\neq n-2$ are equivalent in $S_n$ under the $\{213, 132, 231\}$-equivalence. (The letter $j$ is not fixed here.)
\end{lem}
\begin{proof}
Let $w$ be such a permutation in $S_{n>3}$. Since $w$ begins with $n-1$, ends with $n$, and is of size $\geq 4$, the first three letters of $w$ form a $312$ or $321$ and thus don't form a hit. Since $w$ is non-avoiding and cannot begin with a hit, the right-most $n-1$ letters of $w$ must be non-avoiding. Thus they can be rearranged as the identity with $n-2$ and $n-3$ swapped (by Lemma \ref{lemob1}).
\end{proof}

\begin{lem}\label{lemob3} Let $n>3$. All non-avoiding permutations of the form $(n-1)\ldots n$ are equivalent in $S_n$ under the $\{213, 132, 231\}$-equivalence.
\end{lem}
\begin{proof}
We will prove this inductively. Assume it is true for $S_{n-1}$. The base cases of $S_4$ and $S_5$ can be easily checked computationally. By Lemma \ref{lemob2}, all such permutations whose second letters are not $n-2$ are equivalent. Let $w$ be such a permutation whose second letter is $n-2$. Note that $w$ has the form $(n-1)(n-2) \ldots n$. By the inductive hypothesis (applied to the right-most $n-1$ letters), $w$ is equivalent to any non-avoiding permutation of the form $(n-1)(n-2) \ldots n$.

One such permutation is $y=(n-1)(n-2)\ldots 1n$. Here, the letters not shown are the letters between $1$ and $n-2$ and are in increasing order from left to right. Through repeated applications of $213 \rightarrow 132$, we can reach $w' = (n-1)(n-2)1n \ldots$. From $w'$, we reach $w''=(n-1)n1(n-2)\ldots$ in the following manner.
$$(n-1)(n-2)1n\ldots \rightarrow$$
$$(n-1)1n(n-2)\ldots \rightarrow$$
$$(n-1)n1(n-2)\ldots \phantom{\rightarrow}$$

From $w''$, we may first slide $n-2$ to the final position through repeated applications of $132 \rightarrow 213$, and then slide $n$ to the final position through repeated applications of $132$ or $231$ $\rightarrow 213$. This yields a permutation which is non-avoiding, begins with $n-1$, ends with $n$, and has its second letter $\neq n-2$. Thus all non-avoiding permutations of the form $(n-1)(n-2) \ldots n$ are equivalent to a non-avoiding permutation of the form $(n-1)j \ldots n$ for some $j \neq n-2$. Hence, by Lemma \ref{lemob2}, all non-avoiding permutations of the form $(n-1)\ldots n$ are equivalent.
\end{proof}

\begin{defn}
A permutation in $S_n$ will be called \emph{fronted} if it starts
either with $n-1$ or with $jn(n-1)$ for some $j\leq n-2$.
\end{defn}

\begin{lem}\label{lem8} Any permutation starting with $n-1$ is only equivalent to permutations where $n-1$ is in the first or third position under the $\{213, 132, 231\}$-equivalence.
\end{lem}
\begin{proof} 
It is easy to see that any transformation, applied to a
fronted permutation, yields another fronted permutation. Hence, any
fronted permutation is equivalent only to fronted permutations. In
particular, any permutation starting with $n-1$ is only equivalent to
fronted permutations, and these  have $n-1$ either in the first
or in the third position.
\end{proof}

\begin{lem}\label{lem9} Let $n>3$. Non-avoiding permutations not beginning with $n$ in $S_n$ fall into $2$ non-trivial classes under the $\{213, 132, 231\}$-equivalence.
\end{lem}
\begin{proof}
From an arbitrary non-avoiding permutation not beginning with $n$, we can move $n$ to the right-most position through repeated applications of $132$ or $231$ $\rightarrow 213$. By Lemmas \ref{lemob1} and \ref{lemob3}, we have established that non-avoiding permutations not beginning with $n$ break into at most two classes: those containing permutations of the form $(n-1)\ldots n$, and those containing permutations not beginning with $n-1$ but ending with $n$. Noting Lemma \ref{lem8}, and that there exists a non-avoiding permutation of the form $(n-1)\ldots n$ as well as one which is non-fronted and ending with $n$, they break into exactly two classes in $S_n$ for $n>3$.
\end{proof}

\begin{prop}\label{prop2} $f(n)=f(n-1)+2$ for $n>3$.
\end{prop}
\begin{proof}
We will calculate $f(n)$. Consider the non-avoiding permutations that have $n$ as the first letter. These clearly break into $f(n-1)$ nontrivial classes because the first letter can be ignored. By Lemma \ref{lem9}, the remaining non-avoiding permutations fall into two classes.
\end{proof}

\begin{lem}\label{prop3lem} If the $n$-th letter of an avoiding permutation in $S_n$ is smaller than the $(n-1)$-th one, then the permutation must be the decreasing permutation.
\end{lem}
\begin{proof}
Let $w \in S_n$ be an avoiding permutation. Whenever, for some $i\in\left\{2,3,\ldots ,n-1\right\}$, the $(i+1)$-th letter of $w$ is smaller than the $i$-th one, it is clear that the $i$-th letter must be smaller than the $(i-1)$-th one (since otherwise it would give a $132$ or a $231$ hit, contradicting the avoidance). By applying this observation iteratively, we see that if the $n$-th letter of $w$ is smaller than the $(n-1)$-th one, then the $(n-1)$-th one must be smaller than the $(n-2)$-th one, which in turn must be smaller than the $(n-3)$-th one, etc.. Altogether, this yields that, if the $n$-th letter of $w$ is smaller than the $(n-1)$-th one, the permutation $w$ must be decreasing.
\end{proof}

\begin{prop}\label{prop3} $g(n)=g(n-1) \cdot 2-1$ for $n>3$.
\end{prop}
\begin{proof}
We will calculate $g(n)$ for $n>3$. Consider the avoiding permutations in $S_n$ that have $n$ as the first letter. There are clearly $g(n-1)$ of them because the first letter can be ignored. Now, consider the permutations that have $n$ not as the first letter but are avoiding. Then, $n$ must be the final letter because otherwise there is a $132$ or $231$ in the permutation. Ignoring the letter $n$, these permutations are avoiding permutations from $S_{n-1}$. The converse of this, however, is not true, but instead we have something more delicate: If we take an arbitrary permutation $w$ from $S_{n-1}$ and append $n$ to the end, then the resulting permutation is still avoiding if and only if the $(n-1)$-th letter of $w$ is greater than the $(n-2)$-th one. This requirement holds in all but one case (since Lemma \ref{prop3lem}, applied to $S_{n-1}$, shows that if the $(n-1)$-th letter of an avoiding permutation in $S_{n-1}$ is smaller than the $(n-2)$-th one, then this permutation must be $(n-1)(n-2) \ldots 1$). Hence, $g(n)=g(n-1) \cdot 2-1$ for $n>3$.
\end{proof}

\begin{thm}
The number of classes created in $S_n$ under the $\{213, 132, 231\}$-equivalence is $2^{n-2}+2n-4$.
\end{thm}
\begin{proof} Let $n>3$, noting that the case of $n=3$ is trivial. Because $f(3)=1$, Proposition \ref{prop2} shows that $f(n)=2 \cdot n-5$. Because $g(3)=3$, Proposition \ref{prop3} shows that $g(n)=2^{n-2}+1$. Thus the number of classes created in $S_n$ is $2 \cdot n-5+2^{n-2}+1=2^{n-2}+2n-4$ for $n>3$.
\end{proof}

\subsection{\mathheader{ \{213, 132, 231, 312\} }-Equivalence, \mathheader{ \{123, 132, 213, 231 \} }-Equiva\-lence, and \mathheader{ \{123, 132, 231, 321 \} }-Equivalence} \label{sub4s}
In this subsection, we reference results from Section \ref{Secgen}.

\begin{prop} There are two classes in $S_n$ under the $\{123, 132, 231, 321 \}$-equivalence for $n \ge 4$, one containing permutations equivalent to the identity, and one containing permutations equivalent to the identity with $1$ and $2$ swapped.
\end{prop}
\begin{proof}Note that the position parity of $1$ is invariant under the equivalence. Let $x_n$, $y_n$ be the identity and the identity with $1$ and $2$ swapped in $S_n$. Assume the result inductively, with an inductive base case of $S_4, S_5$. It is not hard to see that in $S_{n>5}$, every permutation is middled. By the inductive hypothesis along with the noted invariance, we may conclude that $L_{n-1}=R_{n-1}=\{x_{n-1}, y_{n-1}\}$. By Theorem \ref{thmstoogesort}, this implies that each permutation in $S_n$ is equivalence to $x_n$ or $y_n$. This along with the noted invariance implies the proposition.
\end{proof}

\begin{prop} There are $n$ classes created in $S_n$ under the $\{123, 132, 213, 231\}$-equivalence.
\end{prop} \label{propsimplen}
\begin{proof}
\end{proof}
\begin{proof}
We will prove this by inducting on $n$. The base case of $S_3$ can easily be shown computationally. Assume the result holds for $S_{n-1}$. Note that if the left-most letter of a permutation is $n$, then $n$ is stationary under the relation. Hence, the permutations starting with $n$ are clearly broken into $n-1$ classes by the inductive hypothesis. Proposition 4.1 of \cite{PRW11} shows that the remaining permutations are equivalent.
\end{proof}

There is a slightly cooler proof of the above proposition which goes as follows. By Theorem \ref{Thmsame2}, the $\{123, 132, 213, 231\}$-equivalence and the  $\{123, 132, 213, 231\}^{\fourdots}$-equivalence are the same in $S_{n\ge 4}$. By Theorem \ref{thmavoiders}, there are $n$ equivalence classes in $S_n$ under the $\{123, 132, 213, 231\}^{\fourdots}$-equivalence for $n\ge 5$ (the smaller cases of $n$ can be shown computationally).

\begin{prop} $S_n$ breaks into three classes under the $\{213, 132, 231, 312\}$-equivalence.
\end{prop}
\begin{proof}
There are always two avoiding permutations. Assume that the remaining permutations are each equivalent in $S_{n-1}$ for $n>5$ (the base cases are easy to check). By Proposition \ref{propreachmiddled}, every non-avoiding permutation in $S_n$ is equivalent to a middled permutation. By Theorem \ref{thmstoogesort}, every non-avoiding permutation in $S_n$ is thus equivalent to the identity with $n$ and $n-1$ swapped.
\end{proof}

\section{Double Replacements}\label{Secdouble}
In this section, we consider the classes created under replacement partitions of $S_3$ with two nontrivial parts, each of size two. Both the forgotten and Knuth relations are members of this family of relations; they are the main inspiration for this direction of work. We find the number of equivalence classes created in $S_n$ in all but one of the unresolved cases (up to symmetry). In the final case, the $\{231, 132\} \{213, 312\}$-equivalence, we provide computational data for the use of future authors. When convenient, we also calculate the size of the class containing the identity.

Surprisingly, class enumerations equal to those yielded by each of the Knuth relation and forgotten relation show up in our study of the $\{123, 132\}\{213, 312\}$-equivalence and $\{123, 231\} \{213, 132\}$-equivalence respectively. The reason for this is still largely a mystery.

The following subsections are concerned with one replacement partition each. In each subsection, the replacement partition $K$ is to be understood to be the partition mentioned in the title of the subsection (unless otherwise specified).

\subsection{\mathheader{\{ 312, 321\} \{ 123, 132 \}}-Equivalence}

\begin{lem}\label{lem21}  (a) Letters to the right of $1$ can be rearranged freely under the \\$\{ 312, 321 \} \{ 123, 132 \}$-equivalence. (b) Letters to the right of $n$ can be rearranged freely under the $\{ 312, 321 \} \{ 123, 132 \}$-equivalence.
\end{lem}
\begin{proof}
We first prove (a). It is sufficient to prove that in $S_n$, all permutations of the form $1 \ldots$ are equivalent to the identity. We will prove this by induction. The base case of $S_3$ is trivial. Assume that we have shown the result to hold for $S_{n-1}$. Let $w$ be a permutation beginning with $1$ in $S_n$. We will prove that $w$ is reachable from the identity. We consider two cases:
\begin{enumerate}
\item The final letter of $w$ is $2$.
\item The final letter of $w$ is not $2$.
\end{enumerate}

We deal with Case 2 first. In this case, the letter $2$ is among the first $n-1$ positions of $w$. Applying the inductive hypothesis to the first $n-1$ letters of $w$, we rearrange the first $n-1$ letters as the identity, yielding $w'$. Note that the first two letters of $w'$ are $12$. Thus applying the inductive hypothesis to the final $n-1$ letters of $w'$, we arrive at the identity.

Now, we consider Case 1. In this case, the final letter of $w$ is $2$. By the inductive hypothesis, we can rearrange the first $n-1$ letters of $w$ to be the identity with the final two letters swapped. We apply $321 \rightarrow 312$ to the final three letters, yielding a permutation starting with $1$ but not ending with $2$. Hence, we can proceed as in Case 2, and conclude that in Case 1, $w$ is equivalent to the identity.

Now, we note that (b) falls from (a) because whenever two permutations $x$ and $y$ are equivalent under the $\{312, 321\}\{123, 132\}$-equivalence, so are their complements. (The \emph{complement} of a permutation $a_1 a_2 \ldots a_n$ is defined as the permutation $(n + 1 - a_1 )(n + 1 - a_2 ) \ldots (n + 1 - a_n )$.)
\end{proof}

\begin{defn}
The \emph{proximum} of a word $w$ is the left-most of the largest and smallest letters in $w$.
\end{defn}

For example, the proximum of $519234$ is $1$. 

\begin{defn}
Let $w$ be a word consisting of $n$ pairwise distinct letters. We define the set $W_w$ (of letters) in the following way.
\begin{itemize}
\item If $w$ has only one letter, then $W_{w}=\varnothing$.
\item Otherwise, $W_{w}=\{  u \} \cup W_{f}$, where $u$ is the proximum of $w$, and $f$ is the factor of $w$ going from the first letter of $w$ to $u$ (inclusive).
\end{itemize}
\end{defn}

For example, if $w=453216$, then $W_w$ contains $1$ as well as the  elements of $W_{45321}$ which are 5 as well as the elements of $W_{45}$, which are just 4. So, $W_w$ contains $1$, $5$, and $4$.

\begin{lem}\label{lemWclosure}
If $w$ and $w'$ are separated by a single transformation, then $W_w = W_{w'}$.
\end{lem}
\begin{proof}
Assume that the lemma holds for smaller $n$, with an inductive base case of $n=1$. Let $w$ and $w'$ in $S_n$ be separated by a single transformation using the hit $h$. We notice that the relative order of $1$ and $n$ can never change during a transformation (since this would only be possible if $1$ and $n$ are in the same hit, but then they would have to be acting as the $1$ and the $3$ of that hit). Hence, either $1$ is to the left of $n$ in each of $w$ and $w^{\prime}$, or $n$ is to the left of $1$ in each of $w$ and $w^{\prime}$. We will only treat the former of these cases; the latter is completely analogous. Assume that $1$ is to the left of $n$ in $w$ and $w'$. Let $f$ and $f'$ be the factors of $w$ and $w'$ respectively which begin with the first letter of $w$ and $w'$ respectively, and end with $1$. If $h$ only involves letters to the left of $1$, then $W_w=W_f \cup \{1\}=W_{f'} \cup \{1\}=W_{w'}$ because $W_f=W_{f'}$ by the inductive hypothesis. If $h$ only involves letters to the right of $1$, then the lemma is trivial. If $h$ involves $1$, then the greatest letter to $1$'s left does not change under the transformation. Hence, the greatest letter in $f$ and $f'$ is the same. Since the letters to the left of that letter in $f$ and $f'$ clearly are static under the transformation, $W_f=W_{f'}$ and $W_w=W_f \cup \{1\}=W_{f'} \cup \{1\}=W_{w'}$.
\end{proof}

\begin{defn}
The \emph{origin permutation} of a permutation $w$ in $S_n$ is a permutation in $S_n$ beginning with the letters of $W_w$, in the same order that they appear in $w$, and then continued with the remaining letters in increasing order.
\end{defn}

\begin{lem}\label{safetylem}
Let $w$ and $w'$ in $S_n$ be such that $W_w=W_{w'}$. Then, the origin permutations of $w$ and $w'$ are the same.
\end{lem}
\begin{proof}
As a consequence of the definition of $W_w$, for a given choice of letters to be in the set, there is exactly one possibility for the order of those letters in $w$. One can find this order in the following way. The right-most letter in $W_w$ is $1$ or $n$. Assume without loss of generality that $W_w$ contains $n$. Then, the next right-most letter is the smallest letter in $W_w$, the next right-most letter is the second largest letter in $W_w$, the next right-most letter is the second smallest letter in $W_w$ and so on. This shows that the order in which the letters of $W_{w}$ appear in $w$ is
uniquely determined by the set $W_{w}$. But the origin permutation of $w$ only depends on the set $W_{w}$ and the order in which the letters of $W_{w}$ appear in $w$. Hence, the origin permutation of $w$ is uniquely determined by the set $W_{w}$, and the origin permutations of $w$ and $w^{\prime}$ are the same.
\end{proof}

\begin{lem}\label{lemWcomplete}
Let $w\in S_n$ and $w'$ be the origin permutation of $w$. Then, $w \equiv w'$  under the $\{ 312, 321\} \{ 123, 132 \}$-equivalence.
\end{lem}
\begin{proof}
Inductively assume the result holds in $S_{n-1}$ with a trivial base case of $S_1$. Let $w$ be in $S_n$ and $w'$ be the origin permutation of $w$. Let $j_1, j_2, \ldots, j_k$ be the letters of $W_w$ in the order that they appear in $w$ from left to right. If $k=1$, then we may apply Lemma \ref{lem21} to $w$ to reach $w'$. Otherwise, we may apply Lemma \ref{lem21} to the factor going from $j_1$ to $j_2$, and slide $j_2$ to be adjacent to $j_1$, reaching $x$ which begins with $j_1j_2$. Applying the inductive hypothesis to the final $n-1$ letters of $x$, we reach $w'$.
\end{proof}

\begin{prop}\label{propWbij}
There are $2^{n-1}$ classes in $S_n$  under the $\{ 312, 321\} \{ 123, 132 \}$-equiva\-lence.
\end{prop}
\begin{proof}
By Lemma \ref{lemWclosure}, each class in $S_{n}$ gives rise to a set $W\subseteq\{1,2,\ldots,n\}  $ such that $W=W_{w}$ for each $w$ in the class, and such that $W$ contains exactly one of $1$ and $n$. Thus we obtain a map from the set of classes in $S_{n}$ to $\{  W \subseteq \{  1,2,\ldots,n\} \ \mid\ W\text{ contains exactly one of }1\text{ and }n\}$. This map is injective as a consequence of Lemma \ref{safetylem} and Lemma \ref{lemWcomplete}. We will now show that it is also surjective.

Let $W$ be a set containing letters from $1$ to $n$, including exactly one of $1$ and $n$. Assume without loss of generality that $W$ contains $n$. Let $k$ be the size of $W$. We define the word $f$ as the word of size $k$ ending with the largest letter in $W$, following the smallest letter, following the next largest letter, following the next smallest letter, etc.. Then, we define $w$ to be the permutation in $S_n$ beginning with $f$ and followed by the letters not in $f$ in increasing order. Note that $W_w=W$.  Therefore, there is a bijection from the set of classes in
$S_{n}$ to the set \newline$\{  W\subseteq\{  1,2,\ldots,n\} \ \mid\ W\text{ contains exactly one of }1\text{ and }n\}  $, and thus the number of classes in $S_{n}$ is
\[
\left\vert \{  W\subseteq \{  1,2, \ldots ,n\}  \ \mid\ W\text{ contains exactly one of }1\text{ and }n\}  \right\vert =2^{n-1}.
\]
\end{proof}

\begin{defn}
Let $w\in S_n$. Let $g_1, g_2, \ldots, g_k$ be the letters of $W_w$ in the order in which they appear in $w$ from right to left. Then, $g_i$ is a \emph{valley} if it is less than each of its adjacent letters in the sequence $g_1, g_2, \ldots, g_k$ and is a \emph{peak} if it is greater than its adjacent letters in the sequence $g_1, g_2, \ldots, g_k$. If $k=1$, then if $g_1=1$ it is a valley and if $g_1=n$ it is a peak. If $n=1$, we consider $g_1$ to be a valley.
\end{defn}

\begin{prop}\label{corsize}
Let $w \in S_n$. Let $g_1, g_2, \ldots, g_k$ be the letters of $W_w$ in the order that they appear in $w$ from \textbf{right to left}. Let $j_i=g_i$ if $g_i$ is a valley and $j_i=n-g_i$ if $g_i$ is a peak. The class containing $w$ is of size $$\dfrac{(n-1)!}{\prod\limits_{a=1}^{k-1}{(j_a+j_{a+1})}}$$ under the $\{ 312, 321\} \{ 123, 132 \}$-equivalence.
\end{prop}
\begin{proof}

We want the number of permutations $w'\in S_n$ with $W_{w'}=W_w$ for a given $w \in S_n$. Recall that in such a situation, $w$ and $w'$ must have the letters of $W_w$ in the same order. We will prove the result by inducting on $k$; the base case of $k=1$ falls from Lemma \ref{lem21}. Assume the corollary holds for smaller $k$. Let $w \in S_n$, $g_1, \ldots, g_k$, and $j_1, \ldots, j_k$ be as stated. Note that $g_k$ is the left-most letter of $w$ and of $w'$. Also note that in this proof, a letter $a$ is said to be \textit{between} two letters $b$ and $c$ if and only if $\min\{b,c\}  <a<\max\{b,c\}  $. In particular, ``between'' refers to value, not position, and does not include the boundaries of the interval.

Let $E$ be the set of all $w^{\prime}\in S_{n}$ such that $W_{w^{\prime}}=W_{w}$. Observe that $E$ is the class of $w$. Let $I$ be the union of $\left\{  g_{k}\right\}  $ with
the set of all letters between $g_{k-1}$ and $g_{k}$. An $I$\emph{-purged word} is a word which contains every letter from $\left\{1,2,\ldots,n\right\}  \setminus I$ exactly once (and no other letters). Let $S$ be the set of all $I$-purged words $x$ with $W_x=\{g_1, g_2, \ldots, g_{k-1}\}$. Let $T$ be the set
  $$\{x\in S_{n-|g_{k}-g_{k-1}|} | W_x= \begin{array}[t]{l} \{g_i | 1\le i \le k-1, g_i \text{ is a valley}\} \; \cup \\
                                                      \{g_i-|g_k-g_{k-1}| | 1\le i\le k-1, g_i \text{ is a peak} \}\} \text{.}
                                   \end{array} $$
 Let $Y$ be the set of words of size $n$ consisting of zeros and a single occurrence of each letter between $g_k$ and $g_{k-1}$ as well as $g_k$ which is the first letter.

A bijection between $S$ and $T$ can be created by mapping $s \in S$ to its order permutation. Hence $|S|=|T|$. 

Let $e\in E$. Let $s$ be the word obtained by striking the letters between $g_k$ and $g_{k-1}$ as well as $g_k$ from $e$. Let $y$ be $e$ except that each letter that is neither between $g_k$ and $g_{k-1}$ nor $g_k$ is replaced with a zero. Then, we define $f: E \rightarrow S \times Y$ such that $f(e)=(s, y)$. Note that $s \in S$ because no $g_{i}$ with $i<k$ is equal to $g_{k}$ or any letter between $g_{k-1}$ and $g_{k}$. Now we go in the other direction; let $s\in S$ and $y\in Y$ be arbitrary. Then, $g: S \times Y \rightarrow E$ is defined such that $g(s, y)$ is $y$ except with each of the zeros of $y$ replaced by the letters of $s$ (in the order that they appear in $s$). Since $f$ and $g$ are inverses of each other, we have that $|E|=|S \times Y|$. Since $|S|=|T|$, we have $|E|=|T \times Y|$.

It is easy to see that $|Y|=(n-1)(n-2)\ldots (n-|g_k-g_{k-1}|+1)$. By the inductive hypothesis, $$|T|=\dfrac{(n-|g_k-g_{k-1}|-1)!}{\prod\limits_{a=1}^{k-2}{(j'_a+j'_{a+1})}}$$ where $j'_{i\le k-1}$ is $g_i=j_i$ if $g_i$ is a valley and is $(n-|g_k-g_{k-1}|)-(g_i-|g_k-g_{k-1}|)=n-g_i=j_i$ if $g_i$ is a peak. Therefore, $$|E|=(n-1)(n-2)\ldots (n-|g_k-g_{k-1}|+1) \cdot \dfrac{(n-|g_k-g_{k-1}|-1)!}{\prod\limits_{a=1}^{k-2}{(j_a+j_{a+1})}}\text{.}$$ This simplifies to $$|E|=\dfrac{(n-1)!}{\prod\limits_{a=1}^{k-1}{(j_a+j_{a+1})}}\text{.}$$
\end{proof}

\begin{prop}
The multiset of sizes of classes in $S_n$ under the  $\{123, 132\} \{213, 231 \}$-equivalence is the same as the multiset of sizes of classes in $S_n$ under the \\ $\{312, 321\} \{123, 132\}$-equivalence.
\end{prop}
\begin{proof}
Let us consider the multiset of sizes of classes in $S_n$ under the $\{123, 132\} \{213, 231 \}$-equivalence. By the induction in the proof of Proposition \ref{prop2n}, it is easy to see that in each class in $S_n$, there is exactly one V-permutation.

Let $v$ be a V-permutation. Consider the set $L_v$ containing the letters to the left of $2$ in $v$ with the exception of the letter $1$. Then, we may define $k-1$ as the size of $L_v$ and $l_1+1, \ldots, l_{k-1}+1$ to be the letters in $L_v$ in increasing order (thus from right to left in position in $v$). Then, by Corollary \ref{corollarysamesizes}, the size of the class containing $v$ is $$\dfrac{(n-1)!}{\prod\limits_{i=1}^{k-1}{l_i}}\text{.}$$ For a given choice of letters $l_1 < l_2 < \cdots < l_{k-1}$ with $l_{k-1}\le n-1$ and $l_1>1$, there are exactly two classes containing a V-permutation $v$ such that $L_v$ consists of the letters $l_i+1$. (The two V-permutations are identical except with $1$ and $2$ swapped.)  

Now, we consider the multiset of sizes of classes in $S_n$ under the $\{312, 321\} \{123, 132\}$-equivalence.

Let $j_1, \ldots, j_k$ be distinct letters from $0$ to $n-1$ such that $j_1$ is $1$ or $0$, $j_2>1$ if $k>1$, $j_{i}<j_{i-2}$ for $i>2$, $j_k+j_{k-1} \le n-1$ if $k>1$, and $k\ge1$. Let $g_i = j_i$ if $i$ is odd and $j_1=1$ or if $i$ is even and $j_1=0$. Let $g_i=n-j_i$ otherwise. Then, since the set $g_1, \ldots, g_k$ contains exactly one of $1$ or $n$, we know that there is exactly one class in $S_n$ such that $W_w=\{g_1, \ldots, g_k\}$ and that it is of size $$\dfrac{(n-1)!}{\prod\limits_{a=1}^{k-1}{(j_a+j_{a+1})}}\text{.}$$

Given a permutation $w$ in $S_n$, we can construct $g_1, \ldots, g_k$ and $j_1, \ldots, j_k$ for some $k \ge 1$ as done in Proposition \ref{corsize} so that they meet the stated restrictions and so that $w$ is in a class of size $$\dfrac{(n-1)!}{\prod\limits_{a=1}^{k-1}{(j_a+j_{a+1})}}\text{.}$$ 

Thus there is a bijection between sets $\{j_1, \ldots, j_k\}\subseteq\{1,2,\ldots,n\}$ for some $k \ge 1$ such that $j_1$ is $1$ or $0$, $j_2>1$ if $k>1$, $j_{i}<j_{i-2}$ for $i>2$, and $j_k+j_{k-1} \le n-1$ if $k>1$ and classes in $S_n$ which are of size $$\dfrac{(n-1)!}{\prod\limits_{a=1}^{k-1}{(j_a+j_{a+1})}}\text{.}$$ 

For such a set $J$, let $L$ be the set containing $l_i=j_i+j_{i+1}$ for $i<k$. Since $j_i<j_{i+2}$, $l_i<l_{i+1}$ for $i<k-1$. Since $j_{k-1}+j_k \le n-1$, $l_i$ is from $1$ to $n-1$ for $i<k$. If $L$ is not empty, then since $a_2>1$, $l_1>1$. For any set $L\subseteq\{2,3,\ldots,n-1\}  $ (possibly empty), there are two possible (and valid) $J$ which could yield such a set $L$, one with $j_1=1$ and one with $j_1=0$. Thus for a given choice of letters $l_1 < l_2 < \cdots < l_{k-1}$ with $l_{k-1}\le n-1$ and $l_1>1$, there are exactly two classes of size $$\dfrac{(n-1)!}{\prod\limits_{i=1}^{k-1}{l_i}}\text{.}$$

Hence, the multiset of sizes of classes in $S_n$ is the same for both equivalences.
\end{proof}

\subsection{\mathheader{\{132, 312 \} \{321, 213 \}}-Equivalence}

\begin{lem}\label{n-1avoiderslem}
There are $n-1$ trivial classes under the $\{132, 312 \} \{321, 213 \} $-equivalence.
\end{lem}
\begin{proof}
Because $312$ and $213$ are hits, we only need to consider \emph{$\Lambda$-permutations}, that is permutations that increase to $n$ and then decrease until the final letter. Since $321$ is a hit, we only need to consider $\Lambda$-permutations which increase for at least the first $n-1$ letters. There are $n$ such permutations because given an ending letter, the rest of such a permutation can be uniquely determined. Since the final possible hit is a $132$, the final letter can be any letter except for $n-1$. Hence, there are $n-1$ trivial classes.
\end{proof}

\begin{defn}
We say that $w \in S_k$ is in $T_k$ if and only if the letters of $w$ other than $1$ and $2$ are in increasing order, and the letters $1$ and $2$ are adjacent and in $21$ order.
\end{defn}

For example, $T_5=\{21345, 32145, 34215, 34521\}$.

\begin{lem}\label{Tclosurelem}
The elements of $T_n$ form an equivalence class under the \\ $\{132, 312 \}\{321, 213  \}$-equivalence.
\end{lem}
\begin{proof}
Let $w\in T_n$. Clearly, the only possible hits in $w$ are $k21$ and $21j$ for some $k$ and $j$ with the actual letters $2$ and $1$. We can rearrange such a hit as $21k$ and $j21$ respectively. By the definition of $T_n$, the permutation reached is in $T_n$. Furthermore, sliding the letters $21$ in this way, all of the permutations in $T_n$ are reachable from each other.
\end{proof}

Note that any three-letter factor $abc$ with $a>b$ must form a hit. This will be useful in the proof of Lemma \ref{nslidelem}.

\begin{lem}\label{nslidelem}
Let $w\in S_n$ be non-avoiding and not in $T_n$. Then, $w$ is equivalent to a permutation which begins with the letter $n$ under the $\{132, 312 \}\{321, 213  \}$-equivalence.
\end{lem}
\begin{proof}
Assume inductively that the proposition holds for $S_k$ for $k<n$. Inductive base cases of $n\in\{3,4,5\}$ can easily be shown computationally. 

We can assume that the final $n-1$ letters of $w$ are non-avoiding. In fact, if the only hit is in the first three letters, then the hit is either $312$ or $213$ and rearranging it as either a $132$ or $321$ creates a hit to the right. 

We can also assume that the final $n-1$ letters of $w$ do not form a permutation in $T_{n-1}$. In fact, if they do, then they can be rearranged as the identity in $S_{n-1}$ with $1$ and $2$ swapped. Then, the first four letters are either $1324$, $2314$, $3214$, or $k213$ where $k>3$. The third case will never happen because then $w$ would have to be in $T_n$. In each remaining case there is a series of steps with which one can reach a permutation which is non-avoiding in the final $n-1$ letters and for which the final $n-1$ letters do not form a permutation in $T_{n-1}$: $1324 \rightarrow 1432 \rightarrow 4132$, $2314 \rightarrow 2431 \rightarrow 4231$, and $k213 \rightarrow k321 \rightarrow 32k1$.

Noting these two assumptions, we can apply the inductive hypothesis to the final $n-1$ letters of $w$, placing $n$ either in the first or second position of $w$, and obtaining $w'$. Because of the position of $n$, the first $n-1$ letters of $w'$ can not form a permutation in $T_{n-1}$ for $n>5$. Also, because of the position of $n$, the first $4$ letters contain a hit. Therefore, since $n>5$, applying the inductive hypothesis to the first $n-1$ letters, we can place $n$ in the first position.
\end{proof}

\begin{prop}\label{killingnprop}
The number of classes containing permutations in $S_{n>3}$ that are non-avoiding and not in $T_n$ equals the number of classes created in $S_{n-1}$ under the \\$\{132, 312 \}\{321, 213  \}$-equivalence.
\end{prop}
\begin{proof} Let $f(n)$ be the number of classes created in $S_n$. By Lemma \ref{nslidelem}, a non-avoiding permutation in $S_n$ that is not in $T_n$ is reachable from a permutation beginning with $n$. Since $n>3$, every permutation beginning with $n$ is non-avoiding and not in $T_n$; we just need to consider how many classes they fall into. Since we can ignore the first letter, permutations beginning with $n$ fall into at most $f(n-1)$ classes. In a hit involving $n$, no two letters that are not $n$ can change relative order. Hence, the permutations beginning with $n$ fall into at least $f(n-1)$ classes, and we are done.
\end{proof}

\begin{prop} There are $\dfrac{n(n+1)}{2}-2$ classes in $S_{n \ge 3}$ under the $\{132, 312 \}\{321, 213  \}$-equivalence.
\end{prop}
\begin{proof}
We will show that with the step from $S_{n-1}$ to $S_n$, $n$ classes are added to the enumeration. Noting the base case of $S_3$, this will inductively prove the proposition. Let $f(n)$ be the number of classes created in $S_n$. By Lemma \ref{n-1avoiderslem} there are $n-1$ trivial classes. By Lemma \ref{Tclosurelem} and Lemma \ref{nslidelem}, there is one nontrivial class containing no permutations beginning with $n$. By Proposition \ref{killingnprop}, the remaining permutations fall into $f(n-1)$ classes. So, $f(n>3)=f(n-1)+n$ and we are done.
\end{proof}

\subsection{\mathheader{\{ 123, 231 \} \{ 213, 132\}}-Equivalence}

 Let $f(m)$ be the number of classes in $S_m$ under the $\{ 123, 231 \} \{ 213, 132\}$-equivalence.

\begin{lem}\label{lemhill} The permutations of a given parity which end with $n$ and which do not begin with $n-1$ are all equivalent under the $\{ 123, 231 \} \{ 213, 132\}$-equivalence.
\end{lem}
\begin{proof}
We will prove this inductively. The result is simple to show computationally for $n \le 5$. Assume (as an
induction hypothesis) that whenever $x\in S_{n-1}$ ends with $n$ and does not begin with $n-1$, $x$ is equivalent either to the identity or the identity but with $1$ and $2$ swapped. Let $w$ be a permutation in $S_n$ ending with $n$ and not beginning with $n-1$. Applying $123 \rightarrow 231$ and $213 \rightarrow 132$ repeatedly, we place $n-1$ in position $2$ and then $n$ in position $n-1$. Then, applying the inductive hypothesis to the left-most $n-1$ letters, we rearrange them as the identity or the identity with $1$ and $2$ swapped depending on the parity of the permutation that they form, yielding $w'$. Note that the first letter of $w'$ is $1$, $2,$ or $3$. Sliding $n$ from the second to final position to the final position with an application of either $132 \rightarrow 213$ or $231 \rightarrow 123$, we then apply the inductive hypothesis to the right-most $n-1$ letters of $w'$ and get $w''$. Note that $w''$ is the identity in $S_n$ except with the first three letters possibly in any order. Finally, applying the inductive hypothesis to the first $n-1$ letters of $w''$, we reach the identity or the identity with $1$ and $2$ swapped in $S_n$. Hence, all such $w$ of a given parity are equivalent.
\end{proof}

\begin{lem}\label{lemdownhill} A permutation beginning with a decreasing subsequence of consecutive
values starting with $n$ (that is, a permutation beginning with
$n(n-1)(n-2) \ldots (n-k+1)$ for some $k\geq 0$) will always have that same
decreasing subsequence at its start after transformations. In
particular, the longest such subsequence never changes in length under
the $\{123,231\}\{213,132\}$-equivalence.
\end{lem}
\begin{proof}
It is sufficient to note that no hits begin with the largest letter in the hit.
\end{proof}

\begin{defn}
We define a $k$\emph{-hill} to be $k$ consecutively positioned and consecutively valued letters in a permutation in $S_n$, each of which is greater than all of the letters to its right in the permutation which are not $n$. A $k$-hill may also be called a \emph{hill} if $k$ is unknown.
\end{defn}

\begin{defn}
In this subsection, we consider a permutation in $S_n$ to satisfy the property $C_k$ if either
\begin{itemize}
\item the permutation begins with $(n-1)$ and begins with a $k$-hill not followed directly by $n$, or
\item the permutation begins with a $g$-hill followed by the letters $jn$ and then a $(k-g)$-hill for some integer $g$ between $0$ and $k$ inclusive and for some letter $j$, where $n-1$ is the first letter of those in the mentioned hills (meaning that $n-1$ is the first letter of the first hill if $g>0$, and of the second hill if $g=0$).
\end{itemize}
\end{defn}

\begin{rem}
Alternatively, we consider a permutation in $S_n$ to satisfy the property $C_k$ if either it begins with a factor $(n-1)(n-2)(n-3)\ldots (n-k+1)$ not followed directly by $n$ or if it begins with a factor $(n-1)(n-2) \ldots (n-i+1)jn(n-i)(n-i-1)\ldots (n-k+1)$ (where $\ldots$ denotes inbetween letters in decreasing order).
\end{rem}

\begin{lem}\label{lemhillprop} If a permutation satisfies $C_k$, then any equivalent permutation will as well under the $\{ 123, 231 \} \{ 213, 132\}$-equivalence.
\end{lem}
\begin{proof}
Let $w$ be a permutation satisfying property $C_k$. We will show that after a transformation, the resulting permutation will still satisfy $C_k$. There are two cases for $w$:
\begin{enumerate}
\item In the first case, $w$ begins with a $k$-hill not directly followed by $n$. If a transformation does not involve any letters in the $k$-hill, then it is easy to see that the resulting permutation still starts with a $k$-hill not directly followed by $n$. If a transformation does involve letters in the $k$-hill, the transformation must be $(n-k)jn \rightarrow jn(n-k)$ for some $j$. This results in a permutation which starts with a $g$-hill followed by the letters $jn$ and then a $(k-g)$-hill for $g=k-1$. Hence, the yielded permutation still satisfies $C_k$.
\item In the second case, $w$ begins with a $g$-hill followed by the letters $jn$ and then a $(k-g)$-hill for some integer $g$ and letter $j$. There are three subcases for a hit involved in a transformation. The hit can contain only letters to the right of $n$ (subcase 1); contain letters only to the left of $n$ (subcase 2); contain $n$ (subcase 3). In subcase 1, by Lemma \ref{lemdownhill}, after a transformation using the hit, the resulting permutation will still satisfy $C_k$. Subcase 2 can never happen because the letters to the left of $n$ are in strictly decreasing order. In subcase 3, the hit is either $213$ or $132$ where $3$ is the actual $n$, $1$ is the letter referred to as $j$ in the definition of $C_k$, and $2$ is a letter in one of the two hills mentioned in the definition of $C_k$. If the hit is used in an application of $213 \rightarrow 132$, the resulting permutation must begin with a $g$-hill followed by the letters $jn$ and then a $(k-g)$-hill for one lesser $g$ than before and the same $j$. If the hit is used in an application of $132 \rightarrow 213$, then either the resulting permutation begins with a $g$-hill followed by the letters $jn$ and then a $(k-g)$-hill for one greater $g$ than before and the same $j$, or it begins with a $k$-hill not followed directly by the letter $n$. Hence, the resulting permutation satisfies $C_k$. 
\end{enumerate}
\end{proof}

\begin{prop} There are $n^2-3n+4$ classes in $S_n$ under the $\{ 123, 231 \} \{ 213, 132\}$-equivalence.
\end{prop}
\begin{proof} Consider the permutations that begin with $n$ for $n>3$. They clearly fall into $f(n-1)$ classes because we can ignore the $n$. From a permutation not beginning with $n$, $n$ can be moved to the final position through repeated applications of $231 \rightarrow 123$ and $132 \rightarrow 213$. Hence, each of the remaining uncounted equivalence classes contains a permutation which ends with $n$. By Lemma \ref{lemhill}, permutations ending with $n$ and not beginning with $n-1$ fall into two classes, divided by parity. Hence, we need only consider permutations of the form $(n-1)\ldots n$ which are not reachable from any permutation ending with $n$ but not beginning with $n-1$. By Lemma \ref{lemhillprop}, these are exactly the permutations of the form $(n-1)\ldots n$ since such permutations each satisfy $C_1$ and permutations not beginning with $n-1$ and ending with $n$ do not.

Let $x$ and $y$ in $S_n$ both begin with $n-1$ and end with $n$. If $x$ and $y$ are of different parity, then clearly they are not equivalent. Assume that $x$ and $y$ are the same parity. Let $k$ be the largest $k$ such that $x$ begins with a $k$-hill and $k$ is less than $n-1$ (when $k=n-1$, the $k$-hill is followed by $n$). Let $k'$ be the largest $k'$ such that $y$ begins with a $k'$-hill and $k'$ is less than $n-1$. 

If $k\neq k'$, then by Lemma \ref{lemhillprop}, $x$ and $y$ are not equivalent. 

Otherwise, assume that $k'=k$. If $k=n-2$ or $k=n-3$, then it is easy to check that $x$ must equal $y$. Thus there is a total of two classes for the cases of $k=n-2$ and $k=n-3$.

Otherwise, assume that $k<n-3$. Then, $x$ and $y$ begin with the same first $k$ letters, the remaining letters do not begin with $n-k$ (because then $x$ or $y$ would begin with a $(k+1)$-hill), and they do end with $n$. Hence, we can apply Lemma \ref{lemhill} to those remaining letters to conclude that $x$ and $y$ are equivalent.

Thus $x$ is equivalent to $y$ exactly when they are of the same parity and they satisfy $C_j$ for the same $j$. When the highest such $j$ is between $1$ and $n-4$ inclusive, there are $2(n-4)$ resulting classes. As mentioned previously, there is one additional class for each of the cases where the highest such $j$ is $n-2$ and $n-3$. Note also that it is impossible for a permutation to satisfy $C_{n-1}$. This means that permutations beginning with $n-1$ and ending with $n$ fall into $2(n-4)+2$ classes. Recall that permutations beginning with $n$ fall into $f(n-1)$ classes, and that the permutations ending with $n$ but not beginning with $n-1$ fall into two more classes. So, $f(n>3)=f(n-1)+2(n-4)+2+2$, $f(3)=4$. This recursion inductively implies that there are $n^2-3n+4$ classes.
\end{proof}

\begin{prop}
The number of elements in the class containing the identity in $S_n$ is $((n-2)(n-1)!)/2$ for $n>3$ under the $\{ 123, 231 \} \{ 213, 132\}$-equivalence.
\end{prop}
\begin{proof}
Consider permutations in $S_{n>3}$. By Lemma \ref{lemhillprop}, no permutations equivalent to the identity satisfy $C_1$. A permutation $w$ fails to satisfy $C_1$ exactly when $w$ begins with neither $n-1$ nor $jn(n-1)$ for any $j$, or when $w$ begins with $(n-1)n$. If a permutation is equivalent to the identity, then it also can not begin with $n$. There are $(n-2)(n-1)!-(n-2)!$ permutations not beginning with $n-1$, $jn(n-1)$, or $n$.  There are $(n-2)!$ permutations beginning with $(n-1)n$. It is easy to see that half of the permutations counted in each of these enumerations are odd and half are even. If we consider only the even ones (noting that transformations do not change parity), there are at most $(n-2)(n-1)!/2$ permutations that can be equivalent to the identity.  From an even permutation not starting with $n$ or $n-1$ and not starting with three letters of the form $jn(n-1)$ for some $j$, through repeated applications $132 \rightarrow 213$ and $231 \rightarrow 123$, one can reach an even permutation ending with $n$ and not starting with $n-1$ which by Lemma \ref{lemhill} is equivalent to the identity. Similarly, all even permutations of the form $(n-1)n\ldots$ are equivalent to the identity. Thus the number of elements in the class containing the identity is $(n-2)(n-1)!/2$ for $n>3$.
\end{proof}

\subsection{\mathheader{\{ 123, 132\} \{ 231, 312 \}}-Equivalence}

\begin{lem}\label{lem1tail} The class containing the identity contains exactly the permutations beginning with $1$ under the $\{ 123, 132\} \{ 231, 312 \}$-equivalence.
\end{lem}
\begin{proof} We will prove this inductively. In $S_3$, the lemma is trivial. Assume the lemma holds in $S_{n-1}$. Let $w$ be a permutation beginning with $1$ in $S_n$. We rearrange the first $n-1$ letters of $w$ as the identity (using the inductive hypothesis), resulting in $w'$. If $w'$ ends with $2$ (i.e., does not have $2$ in the second position), then through repeated applications of $231 \rightarrow 312$ (where the actual two is the lowest letter in the hit), we move it to the third position; then, applying $132 \rightarrow 123$ to the first three letters, we arrive at a permutation starting with $12$. Thus $w$ is equivalent to a permutation starting with $12$, and applying the inductive hypothesis to the final $n-1$ letters of such a permutation, we reach the identity.
\end{proof}

\begin{lem}\label{lemVperm} Every permutation is reachable from a permutation of the form $w1 \ldots$ where $w$ is a word and $w1$ is avoiding. Note that this means $w1$ is decreasing. Thus by Lemma \ref{lem1tail}, all permutations are reachable from a V-permutation under the \\ $\{ 123, 132\} \{ 231, 312 \}$-equivalence.
\end{lem}
\begin{proof} This is trivial for $n=3$. Assume that the result holds for $n-1$. If $1$ is not in the final position of a permutation, then using the inductive hypothesis, we are done. Let $x$ be a permutation ending with $1$. By the inductive hypothesis (on the first $n-1$ letters), $x$ is reachable from a permutation which is a V-permutation in the first $n-1$ letters and ends with $1$. Let $y$ be such a permutation. If $2$ is not the second to final letter, then the final three letters can be rearranged as $231 \rightarrow 312$ yielding a permutation not ending with $1$ (a case for which we have already shown that the lemma holds). If $2$ is the second to final letter, then the entire permutation is decreasing and we have thus reached a V-permutation.
\end{proof}

\begin{lem}\label{lemstaticstartVperms} A V-permutation of the form $j\ldots$ and a V-permutation of the form $k\ldots$ where $k \neq j$ can not be equivalent under the $\{ 123, 132\} \{ 231, 312 \}$-equivalence.
\end{lem}
\begin{proof} The largest letter to the left of the $1$ does not change under the transformations considered.
\end{proof}

\begin{defn}
Let $E_n$ be the set of permutations in $S_n$ which are the descending permutation except with two consecutive letters swapped. 
\end{defn}

For example, $E_4=\{3421, 4231, 4312\}$.

\begin{lem} \label{lemEclass}
The permutations in $E_n$ form an equivalence class under the \\ $\{ 123, 132\} \{ 231, 312 \}$-equivalence.
\end{lem}
\begin{proof}
We will first show that the set $E_n$ is closed under the transformations considered. Let $w \in E_n$. There are at most two hits in $w$, a $231$ hit and a $312$ hit, each using the pair of letters in increasing order in $w$. After a $231 \leftrightarrow 312$ rearrangement, we have simply swapped which pair of consecutive letters is in increasing order, resulting in a permutation in $E_n$.

What remains to be shown is that all of the permutations in $E_n$ are equivalent. Through repeated applications of the rearrangement $312 \rightarrow 231$, we can go from the descending permutation with $1$ and $2$ swapped to each of the permutations in $E_n$. Hence, the permutations in $E_n$ form an equivalence class.
\end{proof}

\begin{lem}
Each V-permutation of the form $\ldots 21\ldots$ is not equivalent to any other V-permutations under the $\{ 123, 132\} \{ 231, 312 \}$-equivalence.
\end{lem}
\begin{proof}
 Let $w$ be a V-permutation of the form $j\ldots 21\ldots$. Since no hit can begin with $21$, the letters to the left of $1$ are in decreasing order, and hits beginning with $1$ can only be rearranged to form other hits beginning with $1$, we have that all permutations equivalent to $w$ must have the letters leading up to $1$ exactly the same as in $w$. So, $w$ can not be equivalent to any V-permutations other than itself.
\end{proof}

\begin{lem}\label{lemunattachedVperms} Each V-permutation of the form $\ldots 31\ldots$ is not equivalent to any other V-permutations under the $\{ 123, 132\} \{ 231, 312 \}$-equivalence.
\end{lem}
\begin{proof}
Let $w$ be a V-permutation of the form $\ldots 31 \ldots$ with $k$ letters to the left of $1$. We posit that the letters to the left of $1$ in each permutation equivalent to $w$ are either exactly the same as in $w$ and are in decreasing order, or form a permutation in $E_{k+1}$ using the same letters as are to the left of $1$ in $w$ as well as the letter $2$. Assume inductively that this is true for permutations which are $j$ transformations away from $w$, with an inductive base case of $j=0$. We will show that the claim holds for permutations $j+1$ transformations away from $w$, completing the proof. Let $x\in S_n$ be $j$ transformations away from $w$ and $x'\in S_n$ be one transformation away from $x$. There are two cases.
\begin{itemize}
\item The letters to the left of $1$ in $x$ are in decreasing order and are the same letters as in $w$. If $x'$ is reached using a hit to the right of $1$, then the claim is trivial. Otherwise, the rearrangement must be $312 \rightarrow 231$ using the actual letters $1,2,3$. Then, the letters to the left of $1$ in $x'$ form a permutation in $E_{k+1}$ and are exactly the letters to the left of $1$ in $w$ in addition to the letter $2$.
\item The letters to the left of $1$ in $x$ form a permutation in $E_{k+1}$ using the same letters as are to the left of $1$ in $w$ in addition to the letter $2$. If $x'$ is reached by a transformation using letters only to the right of $1$, then the claim is trivial. If $x'$ is reached by a transformation using letters only to the left of $1$, then by Lemma \ref{lemEclass}, the claim holds. If $x'$ is reached using a transformation using the letter $1$, then the transformation is $231 \rightarrow 312$ using the actual letters $1,2,3$, and the letters before $1$ in $x'$ are in decreasing order and are the same as in $w$.
\end{itemize}
\end{proof}

\begin{lem}\label{lemattachedVperms} V-permutations of the form $j\ldots k1\ldots$ as well as the V-permutation of the form $j1 \ldots$ where $j\ge 4$ is fixed and $k$ is any letter $\ge 4$ are equivalent under the $\{ 123, 132\} \{ 231, 312 \}$-equivalence.
\end{lem}
\begin{proof}
 By computation, one can show that $51\ldots$ is reachable from $541\ldots$ (we will refer to this as the hypothesis): $ 5 1 2 3 4 \rightarrow 5 1 3 2 4 \rightarrow 3 5 1 2 4  \rightarrow 3 5 1 4 2  \rightarrow 3 4 5 1 2  \rightarrow 3 5 4 1 2  \rightarrow 3 5 2 4 1  \rightarrow 5 2 3 4 1  \rightarrow 5 2 4 1 3  \rightarrow 5 4 1 2 3$. Consider an arbitrary V-permutation $w$ of the form $j\ldots k1\ldots$ for some $j$ and for some $k \ge 4$. By the hypothesis, this is reachable from a permutation starting with the same decreasing factor which ends with $1$ except with $k$ excluded. By Lemma \ref{lem1tail}, this is equivalent to the V-permutation starting with the same factor which terminates with $1$ as $w$ does except with $k$ excluded. Thus all such V-permutations are equivalent to the V-permutation $j1\ldots$ and we are done.
\end{proof}

\begin{prop} In $S_n$, there are $3 \cdot 2^{n-3} + n - 2$ classes under the $\{ 123, 132\} \{ 231, 312 \}$-equivalence.
\end{prop}
\begin{proof}
For $n=3$ this is trivial. Let $n>3$. By Lemma \ref{lemVperm}, we simply have to count the number of classes in $S_n$ that the V-permutations break into. By Lemma \ref{lemunattachedVperms}, there are $2^{n-2}$ classes with only V-permutations of the form $ \ldots 21\ldots$ and $2^{n-3}$ classes with only V-permutations of the form $ \ldots 31\ldots$. By Lemma \ref{lemattachedVperms}, there are $n-3$ classes with V-permutations of the form $ \ldots k1 \ldots$ where $k>3$. Finally, by Lemma \ref{lem1tail}, there is $1$ class containing only the remaining permutations, those of the form $1 \ldots$. So, there are $2^{n-2}+2^{n-3}+n-3+1= 3 \cdot 2^{n-3} + n - 2$ classes in $S_n$.
\end{proof}

\subsection{\mathheader{\{ 123, 132\} \{ 213, 321 \}}-Equivalence}

\begin{defn}
We call a permutation $w$ \emph{bushy-tailed} if the following is true. Let $x$ be the factor of $w$ containing the letters in $w$ to the left of $1$ as well as $1$. In $x$, the letters in odd positions are in decreasing order, the letters in even positions are in decreasing order, and each of the letters of the same position parity as $1$ is less than the adjacent letter to its right and the one to its left (if there is such a letter). In addition, we ask that the letters to the right of $1$ in $w$ be in increasing order.
\end{defn}

 Note that for a permutation ending with $1$ to form a bushy-tailed permutation is equivalent to it avoiding $123$, $132$, $321$, and $213$. (This is not actually very useful, but is interesting to note.)

\begin{defn}
Let $w$ be a bushy-tailed permutation. A letter in $w$ which is to the left of $1$ or is $1$ is a \emph{$w$-starter}.
\end{defn}

\begin{lem}\label{lemdistinctbushytail}
Let $y$ be a bushy-tailed permutation. Let $x$ be a permutation equivalent to $y$ under the $\{ 123, 132\} \{ 213, 321 \}$-equivalence. In $x$, the $y$-starters are in the same relative order as they are in $y$ (property 1). Furthermore, ignoring each other, they are all left to right minima (property 2). Finally, in $x$, each $y$-starter other than $1$ (the final $y$-starter) is less than each letter between it and the next-in-position $y$-starter (property 3). As a consequence, no two distinct bushy-tailed permutations are equivalent.
\end{lem}
\begin{proof}
Let $x$ be a permutation equivalent to a bushy-tailed permutation $y$ which satisfies properties 1, 2, and 3. We will show that any single rearrangement using a hit in $x$ will result in a permutation for which the properties still are satisfied. Assume a rearrangement exists which swaps two $y$-starter letters $j$ and $k$ (where $k$ is without loss of generality, to the right of $j$ in $x$). There are three cases for such a rearrangement.
\begin{itemize}
\item The rearrangement is $213 \rightarrow 321$. Due to property 1, not all three letters can be $y$-starters. So, in order for $j$ and $k$ to be swapped, $j$ must play the role of $3$ in the hit, and $k$ must play the role of $2$ or $1$. Because the remaining letter is not a $y$-starter, property 2 is not satisfied in $x$, a contradiction.
\item The rearrangement is $321 \rightarrow 213$.  Clearly, not all three letters can be $y$-starters. So, $j$ must play the role of $3$ and $k$ must play the role of either $2$ or $1$. If $k$ plays the role of $1$, then property 3 is not satisfied in $x$, a contradiction. If $k$ plays the role of $2$, then property 3 is again not satisfied in $x$, a contradiction.
\item The rearrangement is either $123 \rightarrow 132$ or $132 \rightarrow 123$. Since not all three letters can be $y$-starters (property 1), in both cases, property 2 is not satisfied in $x$, a contradiction.
\end{itemize}
Assume a rearrangement exists which results in a permutation not satisfying property 2. There are two possible cases for the rearrangement.
\begin{itemize}
\item The rearrangement is $132 \rightarrow 123$ where the letters playing the roles of $1$ and $3$ are $y$-starters and the letter playing the role of $2$ is not (the letter playing the role of $1$ is a $y$-starter by property 2; by property 1, not all three letters can by $y$-starters, and for property 2 to be broken by the transformation, $3$ must be a $y$-starter). However, because $2$ is not a $y$-starter, property 3 is not satisfied in $x$, a contradiction.
\item The rearrangement is $321 \rightarrow 213$ where the letter playing the $3$ is a $y$-starter and one of the remaining letters is not. However, then property 3 is not satisfied in $x$, a contradiction.
\end{itemize}
Assume a rearrangement exists which results in a permutation not satisfying property 3. Let $j$ and $k$ be the first and second $y$-starter respectively which have a letter less than $j$ being brought between them by the rearrangement. Such a letter can not come from the left of $j$ by property 2. Such a letter can not come from the right of $k$ because then $k$ would be playing the role of $1$ in $213 \rightarrow 321$ and property 3 could not be satisfied in $x$, a contradiction. Thus such a rearrangement cannot exist.
\end{proof}

\begin{lem}\label{lembushytail} Every permutation is equivalent to a bushy-tailed permutation under the $\{ 123, 132\} \{ 213, 321 \}$-equivalence.
\end{lem}
\begin{proof}
The lemma is easy to show for $S_3$ and $S_4$. Assume inductively that the result holds in $S_{n-1}$ for $n>4$. By Lemma \ref{lemdistinctbushytail}, every bushy-tailed permutation in $S_{n-1}$ is the lexicographically smallest permutation in its class. It follows, that from a permutation $w \in S_n$, repeatedly rearranging the first $n-1$ letters to form a bushy-tailed permutation, then the final $n-1$ letters, then the first $n-1$ letters, and so on, is a process that must halt (because we cannot keep reaching smaller and smaller permutations in $S_n$ forever). Thus, $w\in S_n$ is equivalent to some permutation whose first and final $n-1$ letters both form bushy-tailed permutations. Conveniently, since $n>4$, such a permutation must be bushy-tailed.
\end{proof}

\begin{cor}\label{cor1tail2} Permutations of the form $1\ldots$ are equivalent. Thus there are $(n-1)!$ elements in the class containing the identity under the $\{ 123, 132\} \{ 213, 321 \}$-equivalence.
\end{cor}
\begin{proof}
It follows from Lemmas \ref{lemdistinctbushytail} and \ref{lembushytail} that every permutation beginning with $1$ is equivalent to the identity in $S_n$.
\end{proof}

\begin{prop}
The number of classes in $S_n$ is the sum of the first $n-1$ Motzkin numbers under the $\{ 123, 132\} \{ 213, 321 \}$-equivalence.
\end{prop}
\begin{proof}
By Lemmas \ref{lemdistinctbushytail} and \ref{lembushytail}, the
number of classes in $S_n$ is exactly the number of bushy-tailed
permutations. Let $f(n)$ be the number of classes in $S_n$ (and thus
the number of bushy-tailed permutations). Then, the number of
bushy-tailed permutations in $S_n$ with $n$ not to the left of $1$ is
clearly $f(n-1)$ because we can just append $n$ to the end of each
bushy-tailed permutation from $S_{n-1}$. Consider a bushy-tailed
permutation $w$ with $n$ to the left of $1$. Let $k$ be the number of
letters to the left of $1$ in $w$. Note that $k$ can be any value
greater than $0$ and less than $n$. There are $\binom{n-2}{k-1}$ ways
to choose these letters without yet fixing their order.

Let $g(k)$ be the number of possible arrangements for a given set of
letters to the left of $1$ in a bushy-tailed permutation with exactly
$k$ letters to the left of $1$. Assume that $k$ is even. Then, each
arrangement of the letters corresponds with a $2 \times k/2$ standard
Young tableau (in English notation); reading from \textbf{right to
  left} in the permutation (starting with the letter preceding $1$),
we fill the first column of the tableau, then the second, and so
on\footnote{i.e., if the letter $1$ is the $g$-th letter, then the
  $i$-th column contains the $g-(2i-1)$-th and $g-2i$-th letter.}. The
number of $2 \times k/2$ standard Young tableaux is $C_{k/2}$ where
$C_{n}$ denotes the $n$-th Catalan number\footnote{This is well
  known. It is also easy to prove by bijecting such standard Young
  tableaux with Catalan paths by means of traversing the cells of the
  Young tableau by increasing letter and using each letter in the
  bottom row to represent a side step and each letter in the top row
  to represent a down step.}. Hence, $g(k)=C_{k/2}$ for even $k$. If
$k$ is odd, then $n$ is forced in the first position, and ignoring it,
we see that $g(k)=g(k-1)=C_{(k-1)/2}$. So,
$$f(n)=f(n-1)+\sum\limits_{k=1}^{n-1}{\binom{n-2}{k-1} \cdot C_{\lfloor k/2 \rfloor}}$$
$$=f(n-1)+\sum\limits_{k=0}^{n-2}{\binom{n-2}{k} \cdot C_{\lceil k/2 \rceil}}$$
$$=f(n-1)+\sum\limits_{k=0}^{\lfloor (n-2)/2 \rfloor}{ \binom{n-2}{2k} \cdot C_{k}} + \sum\limits_{k=1}^{\lfloor (n-1)/2 \rfloor}{ \binom{n-2}{2k-1} \cdot C_{k}}\text{.}$$

Recall that $M_n$ is the number of Motzkin $n$-paths, paths from
$(0,0)$ to $(n,0)$ in the grid $\mathbb{N}\times\mathbb{N}$ using only
steps $U = (1,1)$, $F = (1,0)$ and $D = (1,-1)$. We will enumerate
$M_{n}$. Consider the paths where $F$ is used exactly $n-j$ times for
a given $j$. The uses can be distributed in any of
$\binom{n}{n-j}=\binom{n}{j}$ ways. The remaining $j$ steps must
consist of uses of $U$ and $D$. If $k$ is even, then there are clearly
$C_{k/2}$ arrangements for these steps. Note that $k$ can not be
odd. So, $$M_n=\sum\limits_{k=0}^{\lfloor n/2
  \rfloor}{\binom{n}{2k}C_{k}}\text{.}$$
Hence, $$M_{n-2}=\sum\limits_{k=0}^{\lfloor (n-2)/2
  \rfloor}{\binom{n-2}{2k}C_{k}}\text{.}$$
So, $$f(n)=f(n-1)+M_{n-2}+\sum\limits_{k=1}^{\lfloor (n-1)/2 \rfloor}{
  \binom{n-2}{2k-1} \cdot C_k}\text{.}$$

Consider $\sum\limits_{k=1}^{\lfloor (n-1)/2 \rfloor}{
  \binom{n-2}{2k-1} \cdot C_k}$. We will show that this is the
difference between $M_{n-2}$ and $M_{n-1}$. Note that $M_{n-2}$ is the
number of Motzkin $n-1$-paths that go through point $(n-2, 0)$. We
will now enumerate those Motzkin $n-1$-paths which do \textbf{not} go
through $\left( n-1,0\right)$. (There are $M_{n-1}-M_{n-2}$ such
paths.) In such a path, excluding the final step, there must be one
more use of $U$ than of $D$. So, the number of uses of both combined
must be odd. Let that value be $2h-1$. We pick $2h-1$ of the first
$n-2$ steps to be the steps which are not $F$. These $2h-1$ steps,
with a $D$ appended to the end must form a Catalan path of length
$2h$. Thus there are $C_h$ possibilities for the arrangements of these
steps. Hence, $$M_{n-1}-M_{n-2}=\sum\limits_{k=1}^{\lfloor (n-1)/2
  \rfloor}{ \binom{n-2}{2k-1} \cdot C_k}\text{.}$$ So,
$f(n)=f(n-1)+M_{n-1}$. Inductively, $f(n)$ is the sum of the first
$n-1$ Motzkin numbers.
\end{proof}

\subsection{\mathheader{\{123, 132\} \{213, 312\}}-Equivalence}

In this subsection, $f(n)$ be the number of classes that $S_n$ breaks into under the \\$\{123, 132\} \{213, 312\}$-equivalence.

\begin{defn} We call the \emph{root-permutation} (or \emph{root}) of a permutation $w$ the permutation obtained from $w$ by applying $123 \rightarrow 132$ and $213 \rightarrow 312$ repeatedly to the hit ending with $n$ in order to bring $n$ to the first or second position (We will call this \emph{sliding} $n$).
\end{defn}

We first make two observations. When sliding $n$ in $x$, each transformation only uses letters that were to the left of $n$ before the transformation  (observation 1). Additionally, the letters to the left of $n$ after the transformation were static in the transformation (observation 2).

\begin{lem}\label{lemsolidnslide}
Let $x$ and $y$ be two permutations reachable from each other through a single transformation under the $\{123, 132\} \{213, 312\}$-equivalence. Let $x'$ and $y'$ be the root-permutations of $x$ and $y$ respectively. Then, either $x'$ and $y'$ are each of the form $n\ldots$, or they are each of the form $jn\ldots$ for the same $j$.
\end{lem}
\begin{proof}
Let $x,y,x',y'$ be such permutations. Consider the transformation connecting $x$ to $y$. There are the following cases.
\begin{itemize}
\item The transformation involves a hit using letters only to the right of $n$. In this case, the series of transformations from $x$ to $x'$ is the same as the series from $y$ to $y'$.
\item The transformation uses a hit containing $n$. In this case, $x'$ and $y'$ are clearly the same; the series of transformations used to obtain from $x'$ from $x$ is the same as the one used to obtain $y'$ from $y'$ except plus or minus an extra transformation.
\item The transformation uses a hit involving only letters to the left of $n$. There are two subcases.
\begin{enumerate}
\item The rearrangement connecting $y$ to $x$ is $123 \rightarrow 132$ (without loss of generality, in that order). Let $j$, $k$, and $r$ be the letters playing the roles of $1$, $2$, and $3$ respectively, and $w$ be the factor containing the letters to the left of the hit. In $x$, we may slide $n$ within the factor going from $k$ to the end of $x$ to reach a permutation of the form either $wjkn \ldots$ or $wjn \ldots$ (observation $2$). In the first case, if we continue sliding $n$, we reach a permutation of the form $wjn \ldots$. In $y$, we may slide $n$ within the factor going from $r$ to the end of $y$ to reach a permutation of the form either $wjrn \ldots$ or $wjn \ldots$ (observation 2). In the first of the two cases, if we continue sliding $n$, we reach a permutation of the form $wjn \ldots$.  Hence, noting observation 1, completing the sliding of $n$ in $x$ and $y$ will result in either a permutation beginning with $n$ in both cases or a permutation of the form $jn \ldots$ for the same $j$ in both cases.
\item The rearrangement connecting $y$ to $x$ is $213 \rightarrow 312$ (without loss of generality, in that order). Let $j$, $k$, and $r$ be the letters playing the roles of $1$, $2$, and $3$ respectively, and $w$ be the factor containing the letters to the left of the hit. In $x$, we may slide $n$ within the factor going from $r$ to the end of $x$ to reach a permutation of the form either $wkjrn \ldots$ or $wkjn \ldots$ (observation $2$). In each case, we may continue sliding $n$ to reach a permutation of the form $wn\ldots$ (we go $wkjrn\ldots \rightarrow wkjn \ldots \rightarrow wn \ldots$ and $wkjn \ldots \rightarrow wn\ldots$ respectively). In $y$, we may slide $n$ within the factor going from $k$ to the end of $y$ to reach a permutation of the form either $wrjkn \ldots$ or $wrjn \ldots$ (observation 2). In each case, we may continue sliding $n$ to reach a permutation of the form $wn\ldots$ (we go $wrjkn\ldots \rightarrow wrjn \ldots \rightarrow wn \ldots$ and $wrjn \ldots \rightarrow wn\ldots$ respectively). Hence, noting observation 1, completing the sliding of $n$ in $x$ and $y$ will result in either a permutation beginning with $n$ in both cases or a permutation of the form $jn \ldots$ for the same $j$ in both cases.
\end{enumerate}
\end{itemize}
\end{proof}

\begin{lem}\label{keepstartlem}
Let $x, y \in S_n $ be permutations separated by a single transformation under the $\{123, 132\} \{213, 312\}$-equivalence. Let $x', y'$ be the root-permutations of $x$ and $y$. Then, $x'$ is reachable from $y'$ using hits only to the right of $n$ under the $\{123, 132\} \{213, 312\}$-equivalence.
\end{lem}
\begin{proof}
We use as a base case (the ``hypothesis'') that this holds in $S_5$. One can also easily show that the result holds in $S_{\le 5}$. Assume that $n>5$. Let $x, y, x', y'$ be as described in $S_n$. There are three cases for the hit used to reach $x$ from $y$:
\begin{enumerate}
\item The hit uses only letters to the right of $n$. Then $x'$ and $y'$ are reachable from each other with the same transformation (observation 1).
\item The hit contains $n$. Then $x'=y'$ because the transformation is undone in sliding $n$ to the left in one of $x$ or $y$.

\item The hit uses only letters to the left of $n$. Then, $n$ can be slid within the letters to the right of the hit to reach $a$ and $b$ from $x$ and $y$ respectively. In both $a$ and $b$, $n$ is either immediately to the right of the hit, or to the right of the hit and separated from it by one letter. Consider the factors of each $a$ and $b$ containing the hit and the two letters to its right, $w$ and $w'$ respectively. By the hypothesis, the root-permutations of $w$ and $w'$ are reachable from each other using only hits containing letters to the right of $n$. Hence, sliding $n$ within $w$ and $w'$ of $a$ and $b$ respectively, we reach $a'$ and $b'$ such that the letters to the left of $n$ form exactly the same factor, and the letters to the right of $n$ form permutations that are equivalent in $a'$ and $b'$. So, the root-permutations of $a'$ and $b'$ are reachable from each other by using hits only containing letters to the right of $n$ (observation 1). Noting that the root-permutations of $a$ and $b$ are $x'$ and $y'$, $x'$ is reachable from $y'$ using hits only containing letters to the right of $n$.
\end{enumerate}
\end{proof}

\begin{prop}
$f(n \ge 3)=f(n-1)+(n-1) \cdot f(n-2)$.
\end{prop}
\begin{proof}
Sliding $n$ to the left in a given permutation, we can reach a permutation either of the form $jn\ldots$ or $n\ldots$ for any permutation. By Lemma \ref{lemsolidnslide}, permutations of those forms are not equivalent for distinct $j$. By Lemma \ref{keepstartlem}, permutations of the form $jn \ldots$ fall into $f(n-2)$ classes for a given $j$ and permutations of the form $n \ldots$ fall into $f(n-1)$ classes. So, $f(n \ge 3)=f(n-1)+(n-1) \cdot f(n-2)$.
\end{proof}

\subsection{\mathheader{\{123, 321\} \{ 132, 231\} }-Equivalence}

\begin{defn}
The \emph{fall} of a permutation $x \in S_n$ is the set \newline $\left\{k \in \left\{1,2,\ldots ,n\right\} \mid \text{each letter greater than }k\text{ in }x\text{ has the same position parity as }k\right\}$.
\end{defn}

\begin{prop}\label{withtopprop}
Let $x$ be a permutation containing letters $a$ and $b$ separated by one letter. Let $k$ be the greatest letter of different position parity than $a$ and $b$. Assume that $a$ and $b$ are not both in the fall of $x$. Then, $x$ is equivalent to the permutation which is identical to $x$ except with $a$ and $b$ swapped under the $\{123, 321\} \{ 132, 231\} $-equivalence.
\end{prop}
\begin{proof}
Let $x$ be such a permutation with letters $a$, $b$, and $k$ as described. Note that $k>a$ or $k>b$ because $a$ and $b$ are not both in the fall of $x$. We will prove the proposition by inducting on $n$. Assume, inductively, that the result holds for lesser $n$ with a trivial base case of $n=3$. Due to symmetries of the relation,  we can assume without loss of generality that $a$ and $k$ are to the left of $b$. If $k$ is between $a$ and $b$, then we simply swap them. Otherwise, consider the factor of $x$ beginning with $k$ and ending with $b$. If it is of length less than $n$, then applying the inductive hypothesis to it, we are done. If it is of length $n$, then $n$ must be even. Thus we can slide the letter $1$ to the first or final position by repeatedly swapping it with the letter two positions to its left or right respectively. Applying the inductive hypothesis to the remaining $n-1$ letters, we can swap $a$ and $b$. Then, sliding $1$ back to its previous position by repeatedly swapping it with the letter two positions to its right or left respectively, we reach the desired permutation.
\end{proof}

\begin{prop}
The set of permutations equivalent to a permutation $x$ under the $\{123, 321\} \{ 132, 231\} $-equivalence is exactly the set of all permutations with the letters in their fall being in the same relative order as in $x$ and with all letters having the same position parity as in $x$.
\end{prop}
\begin{proof}
By Proposition \ref{withtopprop}, every permutation in the latter set is equivalent to $x$. Observe that under the relation, any permutation equivalent to $x$ must have each letter be of the same position parity as in $x$. Therefore, such a permutation must also have the same fall as $x$ has. The relation never allows two letters in the fall of a permutation to swap relative positions. Hence, any permutation equivalent to $x$ is in the set.
\end{proof}

\begin{cor}
Let $x$ be a permutation with a fall of size $j$. Then, the size of the class containing $x$ is $\frac{ \lfloor n/2 \rfloor ! \cdot \lceil n/2 \rceil ! }{j!}$ under the $\{123, 321\} \{ 132, 231\} $-equivalence. 
\end{cor}
\begin{proof}
Any two letters of the same position parity not both in the fall of $x$ can be swapped. There would be $ \lfloor n/2 \rfloor ! \cdot \lceil n/2 \rceil ! $ elements in the class if letters both in the fall could be swapped as well. However, keeping in mind that they can not, we must divide by the number of ways they can be sorted, $j!$.
\end{proof}

\begin{cor}
Let $l$ be $\lfloor n/2 \rfloor$ and $h$ be $\lceil n/2 \rceil$. There
are $\sum\limits_{j=1}^{l}{ j! \cdot \binom{n-j-1}{h-1}} +
\sum\limits_{j=1}^{h}{ j! \cdot \binom{n-j-1}{l-1}}$ classes in $S_n$
under the $\{123, 321\} \{ 132, 231\} $-equivalence.
\end{cor}
\begin{proof}
Each class is determined by its fall, the order of the letters in its fall in each permutation in the class, the position parity of the letters in its fall, and the position parity of each of the remaining letters. So, if the letters in its fall are of a given position parity and the fall is of size $j$ (each possibility is summed over in the final equation), then there are $j!$ possible orderings for the letters in the fall. Then, out of the $n-j$ remaining letters, the position parity of the letter $n-j$ is already determined as the opposite of that of $n$, but the other $n-j-1$ letters can be of either position parity, thus yielding the binomial coefficient portion of the equation. 
\end{proof}

\subsection{\mathheader{\{123, 231\} \{ 321, 213 \}}-Equivalence}

Let $i_k$ be the identity in $S_k$ and $u_k$ be the identity in $S_k$ except with $1$ and $2$ swapped.

\begin{defn}
A permutation is \emph{layered} if each letter of one position parity is less than each letter of the other.
\end{defn}

\begin{rem}
Note that when studying the $\{123, 231\} \{ 321, 213 \} $-equivalence,  we will often use the symmetry (specific to the relation) between a permutation and its complement. (The \emph{complement} of a permutation $a_1 a_2 \ldots a_n$ is defined as the permutation $(n + 1 - a_1 )(n + 1 - a_2 ) \ldots (n + 1 - a_n )$.)
\end{rem}

\begin{lem}\label{lemidreach}
Let $x$ be a permutation in $S_{n\ge5}$ containing a factor which forms either $i_5$ or $u_5$. Then, $x$ is equivalent to either $i_n$ or $u_n$ under the $\{123, 231\} \{ 321, 213 \}$-equivalence.
\end{lem}
\begin{proof}
We will prove this by inducting on $n$, using $S_5$ and $S_6$ as our base
cases. (These can easily be checked computationally.) Assume the lemma holds
in $S_{n-1}$. Let $x$ be a permutation as described. Without loss of generality, $x$ contains a factor forming $i_5$ or $u_5$ in its first $n-1$ letters. Applying the inductive hypothesis to the first $n-1$ letters, and then the final $n-1$ letters (reaching a permutation ending with $n$), and then the first $n-1$ letters again, we reach either $i_n$ or $u_n$.
\end{proof}

\begin{lem}\label{lemsilly} Let $n>5$. Let $w$ be a permutation in $S_n$ not equivalent to $i_n$ or $u_n$ under the $\{123, 231\} \{ 321, 213 \}$-equivalence. For symmetry reasons, we may assume without loss of generality that $x$ begins with a descent. Under the transformations considered, all permutations equivalent to $x$ begin with a descent as well under the $\{123, 231\} \{ 321, 213 \}$-equivalence.
\end{lem}
\begin{proof}
One can check computationally that the lemma holds in $S_6$. Let $w \in S_{n>6}$ not be equivalent to $i_n$ or $u_n$, begin with a descent, and be equivalent to a permutation not beginning with a descent. Then, there is some transformation which connects two permutations in $w$'s equivalence class, one of which begins with a descent and one of which begins with an ascent. Applying the inductive hypothesis to the first six letters of such a permutation, we rearrange them to form $i_6$ or $u_6$, and by Lemma \ref{lemidreach}, we can reach either $i_n$ or $u_n$.
\end{proof}

\begin{lem}\label{lemsym}
Let $a$ be the decreasing permutation and $b$ be the decreasing permutation with its first two letters swapped, both in $S_{n>5}$. Then, $a$ and $b$ are each equivalent to a different one of $i_n$ and $u_n$ under the $\{123, 231\} \{ 321, 213 \}$-equivalence.
\end{lem}
\begin{proof}
Computationally, one can easily check that this holds for $n=6$. For $n>6$, we may rearrange the first $6$ letters of $a$ and $b$ to form either $i_6$ or $u_6$. Then, applying Lemma \ref{lemidreach}, we reach either $i_n$ or $u_n$ from each of $a$ and $b$. Noting that that parity is invariant under the relation, we conclude that $a$ and $b$ are each equivalent to a different one of $i_n$ and $u_n$.
\end{proof}

For the rest of this subsection, we ask for the reader to remember that Lemma \ref{lemsym} makes a permutation being equivalent to $i_k$ or $u_k$ the same as a permutation being equivalent to $a$ or $b$ (as defined in the proof of the lemma). This means that when a permutation is equivalent to one of $u_k$ or $i_k$, the permutation which is the same but with each letter $j$ mapped to $k-j+1$ is as well; this argument of symmetry, although nontrivial, will be assumed to work for the rest of the subsection.

\begin{lem} \label{lemhighlayer}
Let $x \in S_{n>5}$ be a layered permutation not beginning with $n$ or $n-1$, but starting with a decrease. Then, $x$ is equivalent to $i_n$ or $u_n$ under the $\{123, 231\} \{ 321, 213 \}$-equivalence.
\end{lem}
\begin{proof} Inductively assume the lemma holds in $S_{n-2}$, using $S_6$ and $S_7$ as base cases (these can be checked computationally). Let $x$ be a permutation as described. Assume $x$ is not equivalent to $i_n$ or $u_n$. We must not be able to apply the inductive hypothesis to the first $n-2$ letters, or else $x$ would be equivalent to $i_n$ or $u_n$ by Lemma \ref{lemidreach}. Hence, $x$ must have $n$ or $n-1$ in the final two positions and start with $n-2$. Since we must also not be able to apply the inductive hypothesis to the final $n-2$ letters, the third letter in $x$ must be either $n-1$ or $n$. So, we know that we can slide the right-most of $n-1$ and $n$ two positions to the left through a $123 \rightarrow 231$ or $213 \rightarrow 321$ rearrangement without moving any other of $n$, $n-1$, or $n-2$. Then, applying the inductive hypothesis to the first $n-2$ letters, and noting Lemma \ref{lemidreach}, $x$ is equivalent to $i_n$ or $u_n$, a contradiction.
\end{proof}

\begin{lem} \label{lemthreeops}
Let $x \in S_{n>5}$. For symmetry reasons, we assume without loss of generality that $x$ begins with a decrease. Then, $x$ is equivalent to a layered permutation starting with $n$ or $n-1$ which starts with a decrease or $x$ is equivalent to either $i_n$ or $u_n$ under the $\{123, 231\} \{ 321, 213 \}$-equivalence.
\end{lem}
\begin{proof} Let $x\in S_n$ be not equivalent to the $i_n$ or $u_n$ and begin with a decrease. We assume inductively that the result holds in $S_{n-1}$ with an inductive base case of $S_6$. Applying the inductive hypothesis to the first $n-1$ letters of $x$, we reach a permutation which is layered in all of its letters but possibly its final letter (otherwise, they could be rearranged as $i_{n-1}$ or $u_{n-1}$, we would be able to apply Lemma \ref{lemidreach}, and $x$ would be equivalent to the identity). Without loss of generality, this permutation begins with $n$, $n-1$, or $n-2$ because its first $n-1$ letters do not form a permutation equivalent to $i_{n-1}$ or $u_{n-1}$ (Lemma \ref{lemhighlayer}). This creates the following three cases.
\begin{enumerate}
\item The first letter is $n$. Applying the inductive hypothesis to
the final $n-1$ letters (Using symmetry, they may be rearranged to be either layered or to form $i_{n-1}$ or $u_{n-1}$; by Lemma \ref{lemidreach}, they must form a layered permutation; by Lemma \ref{lemsilly}, this layered permutation begins with an increase.), we reach a layered permutation starting with $n$.
\item The first letter is $n-1$. Applying the inductive hypothesis to
the final $n-1$ letters (Using symmetry, they may be rearranged to be either layered or to form $i_{n-1}$ or $u_{n-1}$; by Lemma \ref{lemidreach}, they must form a layered permutation; by Lemma \ref{lemsilly}, this layered permutation begins with an increase.), we reach a layered permutation starting with
a $n-1$.
\item The first letter is $n-2$. Applying the inductive hypothesis to
the final $n-1$ letters (Using symmetry, they may be rearranged to be either layered or to form $i_{n-1}$ or $u_{n-1}$; by Lemma \ref{lemidreach}, they must form a layered permutation; by Lemma \ref{lemsilly}, this layered permutation begins with an increase.), we reach a layered permutation starting with a descent and starting with $n-2$. By a Lemma \ref{lemhighlayer}, we reach a contradiction, as such  permutation is equivalent to $i_n$ or $u_n$.
\end{enumerate}
\end{proof}

\begin{lem}\label{Lemmafinalpos}
Let $x \in S_{n>5}$ be a layered permutation beginning with $(n-1)$ and not equivalent to $i_n$ or $u_n$ under the $\{123, 231\} \{ 321, 213 \}$-equivalence. Then, $n$ is in the final position of odd parity.
\end{lem}
\begin{proof}
As a base case, one can show computationally that this holds in $S_6$ and $S_7$. Inductively, it holds in $S_{n-2}$. Let $w$ be a layered permutation in $S_{n>6}$ beginning with $(n-1)$ and not equivalent to $i_n$ or $u_n$. Assume that $n$ is not in the final position of odd parity in $w$. Then, so that we can not apply the inductive hypothesis to the first $n-2$ letters and then apply Lemma \ref{lemidreach} to reach $i_n$ or $u_n$, $n$ must be in the second to final position of odd parity. Furthermore, so that we can not apply Lemma \ref{lemhighlayer} to the final $n-2$ letters and then Lemma \ref{lemidreach} to reach $i_n$ or $u_n$, the third letter of $x$ must be $n-2$. However, applying the inductive hypothesis to the final $n-2$ letters and then using Lemma \ref{lemidreach}, this means we can reach $i_n$ or $u_n$ from $x$, a contradiction. 
\end{proof}

\begin{lem}\label{lemtreeing}
Excluding the class containing $i_n$ and the class containing $u_n$, the classes containing some permutation starting with $n$ in $S_{n>6}$ are exactly the classes with permutations not equivalent to $i_{n-1}$ or $u_{n-1}$ and beginning with an ascent in $S_{n-1}$, except with $n$ appended to the beginning of each permutation in each class, under the $\{123, 231\} \{ 321, 213 \}$-equivalence.
\end{lem}
\begin{proof}
Let $n>6$. Consider a layered permutation in $S_{n-1}$ which is not equivalent to $i_n$ or $u_n$ and which (without loss of generality) starts with an ascent. Appending $n$ to the beginning, we reach a layered permutation $x \in S_n$. In permutations equivalent to $x$, $n$ can never be involved in a hit because by Lemma \ref{lemsilly}, the two letters to the right of it will always be in increasing order under the relation considered. Thus $x$ is a layered permutation which is not equivalent to $i_n$ or $u_n$. Now, consider $y$, an arbitrary layered permutation not equivalent to $i_n$ or $u_n$. By Lemma \ref{lemidreach} and symmetry, the final $n-1$ letters of $y$ must form a layered permutation not equivalent to $i_{n-1}$ or $u_{n-1}$. Hence, if $y$ starts with $n$, then $y$ is one of the permutations found in the manner that $x$ is found, completing the proof.
\end{proof}

Let $s \in S_n$ for odd $n$ be the permutation beginning with $2$ and then followed by consecutive values every two positions until the third to final position, ending with $1$, and then going from right to left, increasing every two positions from the second to final letter. Let $s'$ be $s$ except with each letter $j$ mapped to $n-j+1$.

Let $d \in S_n$ for even $n$ be the permutation starting with $2$ and then followed by consecutively increasing letters every two positions until the second to final position, ending with $1$, and increasing from right to left in the remaining positions. Let $f$ be $d$ except with the final two letters swapped. Let $d'$ and $f'$ be $f$ and $d$ respectively except with each letter $j$ mapped to $n-j+1$.

Under the transformations considered, parity is maintained. Hence, it is worth noting that if $n \equiv 1 \pmod 4$, then $s$ and $s'$ are equivalent only to even permutations; if $n \equiv 1 \pmod 4$ then $s$ and $s'$ are only equivalent to even and odd permutations respectively; if $n \equiv 0 \pmod 4$, then $d$ and $d'$ are only equivalent to even permutations while $f$ and $f'$ are only equivalent to odd ones; if $n\equiv 2 \pmod 4$, then $d$ and $f'$ are only equivalent to odd permutations while $d'$ and $f$ are only equivalent to even ones.

\begin{lem}\label{lemextras}
Consider the layered permutations not equivalent to $i_n$ or $u_n$ under the $\{123, 231\} \{ 321, 213 \}$-equivalence which do not begin with $1$ or $n$ for $n \ge 6$. For odd $n$, there are two of them, $s$ and $s'$, each in a class of size $(n+1)/2$. For even $n$, there are four of them, $d$, $f$, $d'$ and $f'$, each in classes of size $n+1$, $n/2$, $n+1$, and $n/2$ respectively.
\end{lem}
\begin{proof}
Assume the lemma holds in $S_{n-1}$ and $S_{n-2}$ (with inductive base cases of $S_6$ and $S_7$). For simplification, we will consider \textbf{only} the permutations beginning with a decrease and note that this describes the others as well by symmetry (and are separate by Lemma \ref{lemsilly}). By the inductive hypothesis for $S_{n-1}$, and the inductive hypothesis for $S_{n-2}$ working along with Lemma \ref{lemtreeing} (applied using symmetry so that we append $1$ rather than $n-1$), in $S_{n-1}$ there is exactly one layered permutation ending with $(n-1)$ and not equivalent to $i_{n-1}$ or $u_{n-1}$, and two layered permutations with $n-1$ in the second to final position which are not equivalent to $i_{n-1}$ or $u_{n-1}$. Recall that any layered permutation not equivalent to $i_n$ or $u_n$ and not beginning with $n$ must begin with $n-1$ by Lemma \ref{lemhighlayer}. Further noting that the final $n-1$ letters of such a permutation must form a layered permutation\footnote{Every factor of a layered permutation forms a layered permutation.} with $n-1$ in the final position of odd parity (Lemma \ref{Lemmafinalpos}), we may conclude that there are at most $2$ layered permutations with $n$ in the second to final position and $1$ layered permutation with $n$ in the final position in $S_n$ which are not equivalent to $i_n$ or $u_n$. By explicitly describing the classes they determine, we will complete the proof.

Keeping symmetry in mind, it is sufficient to show that $s$ is in a class of size $(n+1)/2$, $d$ is in a class of size $n+1$, $f$ is in a class of size $n/2$, and none of these classes overlap. We will do this by characterizing each class.

The class containing $s$ clearly just contains copies of $s$ except with $1$ slid an even number of positions to its left. So, the class containing $s$ has $(n+1)/2$ elements. For example, in $S_5$ the class contains $25341, 25134$, and $12534$.

From $d$, the only transformation is to slide the third to final letter to the right-most position ($321 \rightarrow 213$), bringing us to $e$. From $e$, $1$ can then be slid two positions to the left ($231 \rightarrow 123$), resulting in $e'$ (without going back to $d$, this is the only transformation possible). From $e'$, one can rearrange the first three letters as a $231$ and then slide $1$ two positions to the right. Other than that, without returning to $e$, from $e'$, one can only slide the $1$ an even number of positions to the left (repeated $231 \rightarrow 123$ rearrangements), and choose if the last hit will form a $123$ or $231$ pattern for each of slide of $1$. This results in $5+2(n-4)/2=n+1$ permutations being in the class. As an example, in $S_8$, the class containing $d$ contains $28374651$, $28374516$, $28371456$, $28371564$, $28475614$, $28137456$, $28137564$, $12837456$, and $1283756$. 

From $f$, one can only slide $1$ to any position of the same position parity, resulting in $n/2$ permutations being in its class. For example, when $n=6$, the class contains $263541, 263154$, and $216354$.
\end{proof}

\begin{prop}
For $n>5$ odd, there are $3n-1$ classes in $S_n$ under the \\$\{123, 231\} \{ 321, 213 \}$-equivalence. For $n>5$ even, there are $3n$ classes. 
\end{prop}
\begin{proof}
Let $n>6$. Permutations equivalent to $u_n$ or $i_n$ (or to $u_n$ or $i_n$ except with each letter $j$ mapped to $n-j+1$ by Lemma \ref{lemsym}) fall into two classes, separated by parity. The remaining permutations fall into the same number of classes as they fall into in $S_{n-1}$ (Lemma \ref{lemtreeing}), with the exception of, when $n$ is even, four additional classes, and when $n$ is odd, two additional classes (Lemma \ref{lemextras}). So, noting the number of classes in $S_6$, for $n>5$ odd, there are $3n-1$ classes, and for $n>5$ even, there are $3n$ classes. 
\end{proof}

 \begin{prop}
For $n > 5$, one can easily count the number of permutations in the class containing $i_n$ or in the class containing $u_n$ under the $\{123, 231\} \{ 321, 213 \}$-equivalence. For the sake of brevity, enumerations are provided inside of the proof.
 \end{prop}
 \begin{proof}
Recall that there are two large classes containing permutations equivalent to $i_n$ or $u_n$ in $S_n$. For $n>6$, these classes each consist of the permutations of a given parity not in the smaller classes which were constructed in the proofs of Lemma \ref{lemextras} and Lemma \ref{lemtreeing}. To find how many permutations are in each of the two, it is sufficient to count how many odd and even permutations are in the smaller classes. Let $oi_n, od_n, ei_n, ed_n$ be the number of permutations in these smaller classes that are odd/even and beginning with an increase/decrease as indicated by the variables. We will keep track of these in a matrix (as shown below). Bearing in mind the construction of the classes provided in Lemmas \ref{lemextras} and \ref{lemtreeing}, these variables can be kept track of in the following manner. Let
\[
H_n=  \left( \begin{array}{cc}
oi_n & od_n \\
ei_n & ed_n   \end{array} \right)\text{.}\] 
If $n\equiv 0 \pmod 4$, then 
 \[
H_{n}=  \left( \begin{array}{cc}
od_{n-1}+n/2 & ei_{n-1}+n/2 \\
ed_{n-1}+n+1 & oi_{n-1}+n+1   \end{array} \right)\text{.}\] 
If $n\equiv 1 \pmod 4$, then 
 \[
H_{n}=  \left( \begin{array}{cc}
od_{n-1} & oi_{n-1} \\
ed_{n-1}+(n+1)/2 & ei_{n-1}+(n+1)/2   \end{array} \right)\text{.}\] 
If $n\equiv 2 \pmod 4$, then 
 \[
H_{n}=  \left( \begin{array}{cc}
od_{n-1}+n+1 & ei_{n-1}+n/2 \\
ed_{n-1}+n/2 & oi_{n-1}+n+1   \end{array} \right)\text{.}\] 
If $n \equiv 3 \pmod 4$, then
 \[
H_{n}=  \left( \begin{array}{cc}
od_{n-1} & oi_{n-1}+(n+1)/2 \\
ed_{n-1}+(n+1)/2 & ei_{n-1}   \end{array} \right)\text{.}\]
 
Keeping these identities in mind, enumerations for the variables fall inductively. Assume as the inductive hypothesis that
 \[
H_{4k+0}=  \left( \begin{array}{cc}
4k^2+3k+4 & 4k^2+3k+4 \\
4k^2+3k-1 & 4k^2+3k-1   \end{array} \right)\text{.}\] 

Note that the hypothesis holds for a base case of $k=2$. Using the identities, we get that 
 \[
H_{4k+1}=  \left( \begin{array}{cc}
4k^2+3k+4 & 4k^2+2k+4 \\
4k^2+5k & 4k^2+5k   \end{array} \right)\text{,}\] 
 and that
 \[
H_{4k+2}=  \left( \begin{array}{cc}
4k^2+7k+7 & 4k^2+7k+1 \\
4k^2+7k+1 & 4k^2+7k+7   \end{array} \right)\text{,}\] 
 and that
 \[
H_{4k+3}=  \left( \begin{array}{cc}
4k^2+7k+1 & 4k^2+9k+9 \\
4k^2+9k+9 & 4k^2+7k+1   \end{array} \right)\text{,}\] 
and that
 \[
H_{4k+4}=  \left( \begin{array}{cc}
4k^2+11k+11 & 4k^2+11k+11 \\
4k^2+11k+6 & 4k^2+3k+6   \end{array} \right)
\] 
\[ 
=  \left( \begin{array}{cc}
4(k+1)^2+3(k+1)+4 & 4(k+1)^2+3(k+1)+4 \\
4(k+1)^2+3(k+1)-1 &4(k+1)^2+3(k+1)-1   \end{array} \right) \text{.}
\] 

Hence, the inductive hypothesis holds for $k+1$. Although we use $S_8$ as our base case, we only do this for aesthetic reasons. We could instead use $S_6$ (for which the proposition holds) as a base case, and it is easy to see that the inductive step would still work.
\end{proof}

\subsection{\mathheader{\{123, 321\} \{ 132, 213\}}-Equivalence}

\begin{defn}
Let $j$ and $k$ be two letters in a permutation $w$. A letter $s$ is a \emph{$j,k$-extreme} if $s$ is positioned between $j$ and $k$ and $s$ is either less than $j$ or greater than $k$.
\end{defn}

\begin{defn}
Let $j$ and $k$ be two letters in a permutation. We say that $j$ and $k$ are a \emph{dangerous} pair of letters when the following requirements are satisfied.
\begin{itemize}
\item We require that $j$ is to the left of $k$, $j<k$, and both $j$ and $k$ are of the same position parity.
\item We require that there are more $j,k$-extremes which are of different position parity than $j$ and $k$ than there are of the same.
\end{itemize}
\end{defn}

\begin{lem}\label{lemdanger}
Let $j$ and $k$ be two letters in a permutation $w$. Let $w'$ be a permutation reached from $w$ through a single $321 \rightarrow 123$ transformation using the hit $h$ in $w$. If $j$ and $k$ do not form a dangerous pair of letters in $w$, then they are not a dangerous pair of letters in $w'$.
\end{lem}
\begin{proof}
Let $j,k,w,w',$ and $h$ be as stated. Because of the definition of dangerous, we only need to consider the case where $j$ is to the left of $k$, $j<k$, and $j$ and $k$ are of the same position parity in $w$. If the hit $h$ does not contain $j$ or $k$, then the lemma is trivial. The hit cannot contain both $j$ and $k$ because $j<k$ and $j$ is to the left of $k$. If the hit contains only one of $j$ and $k$, then by symmetry, we may assume without loss of generality that it contains $j$. Hence, there are three cases.
\begin{enumerate}
\item In this case, $j$ is the first letter of $h$. Then, in $w'$, there are two less $j,k$-extremes (acting as $2$ and $1$ in $h$) than in $w$, one of each position parity. Hence, $j$ and $k$ are not a dangerous pair of letters in $w'$.
\item In this case, $j$ is the second letter of $h$. Then, in $w'$, a $j,k$-extreme of different position parity than $j$ and $k$ (acting as $1$ in $h$) is replaced with a letter which may or may not be a $j,k$-extreme (acting as $3$ in $h$). Hence, there still are still as least as many $j,k$-extremes of the same position parity as $j$ and $k$ as there are of different position parity in $w'$, and $j$ and $k$ are not a dangerous pair of letters in $w'$.
\item In this case, $j$ is the final letter of $h$. Since the transformation using $h$ only adds letters to the set of letters which are between $j$ and $k$, the number of $j,k$-extremes of each position parity can only increase. If a $j,k$-extreme of different position parity than $j$ and $k$ is added (acting as the $2$ in $h$), then so is one of the same position parity (acting as the $3$ in $h$). Hence, $j$ and $k$ are not a dangerous pair of letters in $w'$.
\end{enumerate}
\end{proof}

\begin{defn}
A permutation is \emph{zipped} if for each letter except for the final two letters, the letter two positions to its right is smaller.
\end{defn}

\begin{defn}
A permutation is $k$\emph{-downed} for $k>0$ if every permutation which can be reached through a single $123 \rightarrow 321$ rearrangement is $k-1$-downed. A permutation is $0$-downed if it is zipped.
\end{defn}

\begin{prop}\label{propzippedisolation}
Let $d$ be a downed permutation and $w$ be a permutation equivalent to $d$ under the $\{123, 321\} \{ 132, 213\}$-equivalence. Let $k$ be the number of transformations needed for $d$ to be reached from $w$. Then, $w$ is $k$-downed and both $132$ and $213$ avoiding.
\end{prop}
\begin{proof}
We will prove this by inducting on $k$, with a trivial base case of $k=0$. Assume as the inductive hypothesis, that the proposition holds for all smaller $k$. Let $d$ and $w$ be as described. Since $w$ is reachable from $d$ through some $k$ transformations, let us pick such a sequence of transformations and let $r$ be $w$ after the first transformation which uses the hit $h$. By the inductive hypothesis, $r$ is $k-1$-downed. If $w$ is reached from $r$ through a $123 \rightarrow 321$ transformation, then $w$ is $k-2$-downed, which cannot be. If $w$ is reached from $r$ through a $132 \leftrightarrow 213$ transformation, then $r$ is a $k-1$-downed permutation which is not $132$ and $213$ avoiding, a contradiction. Hence, $w$ is reached from $r$ through a $321 \rightarrow 123$ rearrangement. 

Since $w$ can be reached from $d$ through as sequence of $321 \rightarrow 123$ transformations, by Lemma \ref{lemdanger}, $w$ contains no pair of dangerous letters. Hence, $w$ is $231$ and $213$ avoiding since both patterns contain a pair of dangerous letters.

It remains to show that $w$ is $k$-downed. Because of the inductive hypothesis, it is sufficient to show that any $123 \rightarrow 321$ transformation applied to $w$ brings us to a permutation $r'$ which can be reached from $d$ through $k-1$ transformations\footnote{Note that $r'$ cannot be $i$-downed for $i<k-1$ because then $w$ would not require $k$ transformations to be reached from $d$.}. Let $r'$ be a permutation reached from $w$ through a single $123 \rightarrow 321$ rearrangement using the hit $h'$. If $h'=h$ then $r'=r$ which clearly can be reached from $d$ in $k$ transformations. If $h$ includes just some of the letters of $h'$, then the letters around and including $h$ must form a $1234$ pattern; in $r$, this becomes a pattern which contains a $132$ or $213$ pattern, which cannot be. In the final case, $h'$ and $h$ do not use any of the same letters. Let $x$ be the permutation reached from $r$ by rearranging $h$ and from $r'$ by rearranging $h'$. Since $x$ can be reached from $r$ through a $123 \rightarrow 321$ rearrangement, it is $k-2$-downed, and hence reachable from $d$ in $k-2$ transformations. Thus $r'$ is reachable from $d$ in $k-1$ transformations, and we are done.
\end{proof}

\begin{defn}
A permutation is \emph{partially zipped} if for each letter but the final three letters, the letter two positions to its right is smaller, and the final two letters are $1n$.
\end{defn}

\begin{defn}
A permutation is $k$\emph{-pdowned} for $k>0$ if every permutation which can be reached through a single $123 \rightarrow 321$ rearrangement is $k-1$-pdowned and each permutation reached through a single $1nl \rightarrow l1n$ rearrangement is $k$-pdowned. A permutation is \emph{$0$-pdowned} if it is zipped.
\end{defn}

Note that the preceding definition is both explicitly and implicitly recursive, with two actual recursion parameters. Given the shortest path of transformations from a $k$-downed permutation and a $0$-downed permutation, we use both $k$, the parameter for the number of $123 \rightarrow 321$ rearrangements in this path, and the number of $1nl \rightarrow l1n$ rearrangements in the path as parameters.

\begin{defn}
 A pair of letters in a permutation is \emph{pdangerous} if the pair is dangerous and one of the following is true.
\begin{itemize}
\item Neither letter is $1$ or $n$.
\item Exactly one of the two letters is $1$ or $n$ and the other of $1$ or $n$ is not between the two letters. 
\end{itemize}
\end{defn}

\begin{lem}\label{lempdanger}
Let $w$ be a permutation which contains no pair of pdangerous letters, which has $1$ and $n$ with different position parities and which has $1$ to the left of $n$, and for which the letters positioned between $1$ and $n$ can be paired up so that each letter with the same position parity as $1$ is paired with a smaller letter of the same position parity as $n$. Let $w'$ be a permutation reached from $w$ with a single $321 \rightarrow 123$ rearrangement or a $l1n \rightarrow 1nl$ rearrangement. Then, $w'$ satisfies the properties noted for $w$. 
\end{lem}
\begin{proof}
If the rearrangement going from $w$ to $w'$ does not involve $n$ or $1$, then this falls from Lemma  \ref{lemdanger}. If the rearrangement involves both $n$ and $1$, then it must be $ln1 \rightarrow 1nl$. In this case, the set of letters between $1$ and $n$ does not change; no pair of letters with either $1$ or $n$ can be pdangerous because $1$ and $n$ are adjacent; and any pair of pdangerous letters, $j$ and $k$, remains pdangerous after the transformation because if either $1$ or $n$ changes its state of being a $j,k$-extreme, so does the other which is of different position parity.

In the remaining case, the transformation from $w$ to $w'$ uses exactly one of $n$ or $1$. Let $h$ be the hit in $w$ involved in the transformation. Because of symmetry, we can assume without loss of generality that $h$ uses $n$ as its first letter and is of the form $nab$ in $w$ and $ban$ in $w'$. By Lemma \ref{lemdanger}, no pair of letters in $w'$ which does not include one of $n$ or $1$ can be pdangerous in $w'$ which was not pdangerous in $w$. Pairing up $a$ with $b$, we see that the claim of  ``the letters positioned between $1$ and $n$ can be paired up so that each letter with the same position parity as $1$ is paired with a smaller letter of the same position parity as $n$'' still holds after the transformation. As a consequence, $a$ and $1$ cannot form a pdangerous pair in $w'$ because because there are at least as many $1,a$-extremes with the same position parity as $1$ and $a$ as there are with different position parity; this means there are no pdangerous pairs of letters including $1$ in $w'$.  As an additional consequence, $n$ cannot be in a pdangerous pair in $w'$ because the pair of letters would have the same letters between them as in $w$ except with the addition of $a$ and $b$; since $a$ being a $j,n$-extreme implies $b$ is as well, there are still at least as many $j,n$-extremes with the same position parity as $n$ and $j$ as there are with different position parity. So, $w'$ has no pdangerous pairs of letters.
\end{proof}

\begin{lem}\label{lem1nslide}
Let $w$ be a $k$-pdowned permutation with $1$ and $n$ adjacent. Let $w'$ be $w$ after a $j1n \rightarrow 1nj$ rearrangement. Then, $w'$ is $k$-pdowned.
\end{lem}
\begin{proof}
Let $w$ and $w'$ be as specified. Let $h$ and $h'$ be the hit $w$ and $w'$ respectively which can be rearranged to obtain $w'$ and $w$ respectively. It is trivial to note that rearranging $h'$ in $w'$ yields a $k$-downed permutation. We need to show that any $123 \rightarrow 321$ rearrangement in $w'$ using the hit $g$ results in a $k-1$-pdowned permutation. Assume inductively that the lemma holds for lower $k$, with a trivial inductive base case of $k=0$. Since $w$ is $k$-pdowned, $w$ with $1n$ slid to the right-most position, yielding $x$, is $k$-pdowned as well. Note that $x$ is also equal to $w'$ with $1n$ slid to the right-most position. Since $g$ cannot involve $1$ or $n$, $g$ is still a hit in $x$ using the same letters as in $w'$. Let $x'$ be $x$ but with that hit rearranged. Since $x$ is $k$-pdowned, $x'$ is $k-1$-pdowned. Note that $x'$ can be reached from $w'$ with $g$ rearranged be sliding $1n$ to the end of $w'$ with $g$ rearranged. Therefore, by repeated applications of the inductive hypothesis, we see that $w'$ with $g$ rearranged is $k-1$-pdowned as well, and we are done.
\end{proof}

\begin{prop}\label{proppzippedisolation}
Let $d$ be a partially zipped permutation and $w$ be a permutation equivalent to $d$ under the $\{123, 321\} \{ 132, 213\}$-equivalence. Of the transformations needed for $d$ to be reached from $w$ in the fewest number of transformations possible, let $k$ be the least number of necessary transformations which are not $1nj \rightarrow j1n$. Then, $w$ is $k$-pdowned. Furthermore, $w$ is $132$ and $213$ avoiding except for when $1$ and $n$ act as $1$ and $3$ respectively.
\end{prop}
\begin{proof}
Let $d$ and $w$ be as stated. If $1$ and $n$ are adjacent in $w$, then by Lemma \ref{lem1nslide}, we may assume that $w$ ends with $1n$. We will assume inductively that the lemma holds for smaller $k$, with a trivial base case of $k=0$. Since $w$ is reachable from $d$ through a series of transformations, $k$ of which are not $1nj \rightarrow j1n$, let us pick such a sequence of transformations and let $r$ be $w$ after the first transformation which uses the hit $h$. If $h$ is a $132$ or $213$ pattern that does not use $1$ and $n$ as $1$ and $3$, then $r$ is a $k-1$-pdowned permutation containing a $132$ or $213$ pattern of that form, a contradiction. Note that $h$ cannot be $1nj$ for any $j$ since we already assumed that if $1$ and $n$ are adjacent, they are at the end of $w$.  If $w$ is reached from $r$ through a $j1n \rightarrow 1nj$ transformation, then by the inductive hypothesis, $w$ is $k-1$-pdowned, a contradiction. If $w$ is reached from $r$ through a $123 \rightarrow 321$ transformation, then by the inductive hypothesis, $w$ is $k-2$-pdowned, a contradiction. Hence, $w$ is reached from $r$ through a $321 \rightarrow 123$ rearrangement and by the inductive hypothesis, $r$ is $k-1$-pdowned. 

Since $r$ is $k-1$-pdowned and $r$ is reached from $w$ by a $123 \rightarrow 321$ rearrangement, by repeated applications of Lemma \ref{lempdanger} (noting that the lemma can be applied to any partially zipped permutation), $w$ must meet the requirements set by the lemma. As a consequence, $w$ is $132$ and $213$ avoiding except for when $1$ and $n$ act as $1$ and $3$ respectively.

It remains to show that $w$ is $k$-pdowned. Because of the inductive hypothesis, it is also sufficient to show that any $123 \rightarrow 321$ or $1nj \rightarrow j1n$ transformation applied to $w$ brings us to a permutation $r'$ from which $d$ can be reached through $1nj \rightarrow j1n$ and $123 \rightarrow 321$ transformations, $k-1$ of which are $123 \rightarrow 321$. Recall that we already assumed no $1nj$ hit exists, so we do not need to consider that case. Let $r'$ be a permutation reached from $w$ through a single $123 \rightarrow 321$ rearrangement using the hit $h'$. If $h'=h$ then $r'=r$ which is $k-1$-pdowned. If $h$ and $h'$ do not use any of the same letters, then let $x$ be the permutation reached from $r$ by transforming $h'$ and from $r'$ by transforming $h$; noting the inductive hypothesis, $x$ is $k-2$-pdowned. Hence, $d$ is reachable from $r'$ through $1nj \rightarrow j1n$ and $123 \rightarrow 321$  transformations, $k-1$ of which are $123 \rightarrow 321$. In the final case, $h$ and $h'$ share some but not all of their letters. Since, $h$ and $h'$ share some letters, the letters surrounding $h$ must form the permutation $1234$ where $h$ is the final or first three letters of the permutation. In order for $r$ to be $132$ and $213$ avoiding except for cases where $1$ and $n$ play the role of $1$ and $3$, $1$ and $n$ must play the role of $1$ and $4$ in the $1234$. Hence, in order for $h'$ to exist in $w$, it must contain the other of either the first three or final three letters in the $1234$. Let $x$ be $w$ except with the $1234$ rearranged as $3214$. Either $r$ is $x$ or $r$ with a $1nj \rightarrow j1n$ rearrangement is $x$. Hence, $x$ is $k-1$-pdowned.  Either $r'$ is $x$ or $r'$ with a $1nj \rightarrow j1n$ rearrangement is $x$. Hence, $d$ is reachable from $r'$ through $1nj \rightarrow j1n$ and $123 \rightarrow 321$  transformations, $k-1$ of which are $123 \rightarrow 321$ and we are done.
\end{proof}

\begin{prop}
Let $f(n)$ be the number of classes created in $S_n$ under the \\$\{ 123, 321 \} \{ 132, 213 \}$-equivalence for $n>4$. Then,
$$f(n)=\binom{n}{\lfloor n/2 \rfloor}+\binom{n-2}{\lfloor (n-2)/2 \rfloor}+3\text{.}$$ 
\end{prop}
\begin{proof}
Let $n$ be $\geq 5$. By Proposition \ref{propzippedisolation} and Proposition \ref{proppzippedisolation}, zipped and partially zipped permutations are each in their own classes. This accounts for the first two terms of the formula. We will show that the remaining permutations fall into $3$ classes.

There must be at least three remaining classes because the number of inversions in a permutation modulo $3$ is an invariant in the relation considered. We will now show that all permutations in the remaining classes (for $n \ge 5$) are equivalent to one of $12543678 \ldots$, $21435678 \ldots$, and $12435678 \ldots$. It is not hard to see that these permutations have $3$, $2$, and $1$ inversions respectively, putting them into each of the three prospective classes. Our reason for choosing them is because each has a $132$ hit starting in the second position. This will be useful to us shortly.

Let $w$ be a permutation not equivalent to a zipped or partially zipped permutation. Through repeated applications of $123 \rightarrow 321$, one can reach a $123$ avoiding permutation. Because the permutation is not zipped or partially zipped, it must contain a $132$ or $213$ pattern. If $1$ and $n$ are in the said pattern, we can slide them to the right-most positions and repeat the process on the remaining letters, bringing us to a permutation which contains a $132$ or $213$ pattern which does not use both $1$ and $n$. Hence $w$ is equivalent to some $w'$ which contains a $132$ or $213$ pattern not using $1$ and $n$ as the letters $1$ and $3$. 

If $w'$ has a $132$ or $213$ hit in the first $n-1$ letters such that it does not use both the smallest and largest of the $n-1$ letters or in the final $n-1$ letters such that it does not use both the smallest and largest of the $n-1$ letters, we define $w''$ as $w$. Otherwise, if $w'$ has only one $132$ or $213$ hit which does not use both $n$ and $1$, the said hit is in the right-most position and contains both $1$ and $n$, $n$ is in the first position of $w'$, then arranging the hit as $132$, we reach a $w'$ such that its first $n-1$ letters form a permutation containing a $213$ hit that does not use both the highest and lowest letter. Similarly, if $1$ is in the final position of $w'$ and the only $132$ or $213$ hit not using $1$ and $n$ in $w'$ instead uses $n$ and $2$ and is in the final three positions, then arranging the hit as $213$, we reach a $w''$ such that the final $n-1$ letters contain a $132$ pattern does not use both the highest and smallest letter of those $n-1$ letters.

If $w'$ has a $132$ or $213$ hit in the first $n-1$ letters such that it does not use both the smallest and largest of the $n-1$ letters or in the final $n-1$ letters such that it does not use both the smallest and largest of the $n-1$ letters, we define $w''$ as $w'$. Otherwise, (noting symmetry) $w'$ either begins with $1$ and ends with $2jn$ or $j2n$ or ends with $1$ and begins with $2jn$ or $j2n$ for $2<j<n$. In each of these cases, if we were not able to already define $w''$ simply as $w$, it is not hard to see that a rearrangement of the hit creates a new $132$ or $213$ hit either starting with the second letter or ending with the second to final letter of the new permutation. This yields a $w''$ which has a $132$ or $213$ hit in the first $n-1$ letters such that it does not use both the smallest and largest of the $n-1$ letters or in the final $n-1$ letters such that it does not use both the smallest and largest of the $n-1$ letters.

Assume as an inductive hypothesis for the rest of the proof, that any permutation containing a $132$ or $213$ pattern not using both $1$ and $n$ in $S_k$ for $k<n$ is equivalent to one of $12543678 \ldots$, $21435678 \ldots$, and $12435678 \ldots$. The base cases of $k=5$ and $k=6$ for this induction are easy to check computationally. We will now show that $w''$ is equivalent to one of $12543678 \ldots$, $21435678 \ldots $, and $12435678 \ldots$, completing the proof. If $w''$ has a hit in the first $n-1$ letters not using the highest and smallest letter of those first $n-1$ letters, we apply the inductive hypothesis to them, reaching $x$. Because of the hit that $x$ has in the final $n-1$ letters (starting with the second letter), we can apply the inductive hypothesis to the final $n-1$ letters to reach $x'$ which has $n$ in the final position. Finally, applying the inductive hypothesis to the first $n-1$ letters (using the hit starting in the third position), we are done. If instead, $w''$ initially has a $132$ or $213$ pattern in the final $n-1$ letters not using both of the highest and lowest letters in the final $n-1$ letters, we can apply the inductive hypothesis to the final $n-1$ letters to reach an instance of the case we have already covered.
\end{proof}

\subsection{\mathheader{\{123, 231\} \{213, 312\}}-Equivalence}

We will first count the number of classes which contain only $123$ and $231$ avoiding permutations.

\begin{defn}
A \emph{peak} in a word is a letter that is greater than each of its adjacent letters in the word. A \emph{dip} in a word is a letter that is less than each of its adjacent letters in the word.
\end{defn}

\begin{defn}
The $321$\emph{-leading factor} of a permutation $w$ is the factor containing the letters before the first occurrence of $321$ in $w$ as well as the first letter of that occurrence. If there is no occurrence, then it is simply $w$. 
\end{defn}

\begin{defn}
We consider the $k$\emph{-length} of a permutation to be the number of peaks in its $321$-leading factor.
\end{defn}

\begin{defn}
We define the \emph{$321$-leading segments} of a permutation $w$ to be the factors beginning with the final two letters of a $321$ hit and going to the first letter of the next $321$ hit. The $321$-leading factor as well as the factor going from the second letter of the final $321$ hit to the end of a permutation are considered $321$-leading segments as well.
\end{defn}

\begin{defn}
The $k$\emph{-length} of a $321$-leading segment is the number of peaks in the $321$-leading segment.
\end{defn}

\begin{defn}
We say that a permutation $w$ is \emph{compact} if it satisfies one of the following two conditions.
\begin{itemize}
\item  (condition 1) $w$ begins with a decrease. The $321$-leading factor of $w$ is alternating, with the letters in the odd positions of the factor being the largest letters in $w$. Also, The portion of $w$ not included in the $321$-leading factor, $w'$ either satisfies this condition (recursively) or is of length $0$. Finally, if it is not of length $0$, then where $k$ is the $k$-length of $w$, $k'$ is the $k$-length of $w'$, and $j$ is the value of the letter in the final dip of $321$-leading factor of $w$, we have $k+k' \le n-j$.
\item  (condition 2) $w$ begins with an increase. The final $n-1$ letters of $w$ form a permutation satisfying condition $1$.
\end{itemize}
\end{defn}

\begin{prop}
Each compact permutation is only equivalent to permutations avoiding both $123$ and $231$ under the $\{123, 231\} \{213, 312\}$-equivalence. Each class containing only permutations avoiding both $123$ and $231$ contains exactly one compact permutation. 
\end{prop}
\begin{proof}
We will first show that each compact permutation determines a unique class containing only $123$ and $231$ avoiding permutations. We claim that the class containing a compact permutation $w$ contains exactly the permutations which are $w$ except with the peaks in each $321$-leading segment rearranged arbitrarily. Assume as the inductive hypothesis that the claim holds for permutations with $r$ or less $321$-leading segments (with the base cases of $r=1$ and $r=0$ being fairly obvious). 

Let $w$ be a compact permutation with $r+1$ $321$-leading segments. Let $w'$ be $w$ except excluding the $321$-leading factor. Let $l$ be the $321$-leading factor of $w$ and $l'$ be the $321$-leading factor of $w'$. Let $k$ be the $k$-length of $w$, $k'$ be the $k$-length of $w'$, and $j$ be the value of final dip $l$ (if it does not exist, then the $321$-leading factor of $w$ is of size $1$ and the claim easily falls from the inductive hypothesis). Note that $l$ ends with a peak since otherwise, there would be a $321$-leading factor prior to the first one in $w$. Recall that $w$ is compact. Since all of the dips in $l$ are less than or equal to $j$, all of the peaks in $l$ are among the greatest $k$ letters of $w$, all of the peaks of $l'$ are among the greatest $k'$ letters of $w'$, and $k+k' \le n-j$, each peak in $l'$ is greater than $j$ (observation 1). By the inductive hypothesis, the permutations contained in the equivalence class of $w'$ are simply copies of $w'$ where the peaks in each $321$-leading segment are possibly scrambled (observation 2). Also, by the inductive hypothesis, the permutations contained in the equivalence class of $l$ are simply copies of $l$ except with its peaks arbitrarily scrambled (observation 3). By observations 1, 2, and 3, for any permutation $x$ reachable from $w$ by rearrangements of hits in $l$ and rearrangements of hits in $w'$, the letter immediately following $l$ is a peak in $w'$ which is greater than $j$. We also know the said letter is less than the peak immediately following $j$. Hence, $j$ and the two letters following it form a $132$ pattern in $x$. Since the three letters following $j$ form a $321$ pattern in $x$, there is no hit in $x$ containing both letters from $l$ and $w'$. Therefore, the class containing $w$ contains only versions of $w$ where the peaks within each $321$-leading segment have been scrambled arbitrarily and only permutations avoiding $123$ and $231$. Thus $w$ uniquely determines such a class. 

We will now show that every class $C$ containing only $123$ and $231$ avoiding permutations contains a compact permutation $w$, thus completing the proof. Let $x$ be a permutation in such a class $C$. Applying repeated $213 \rightarrow 312$ rearrangements to $x$, we reach a permutation $w$ consisting only of $321, 132$, and $312$ patterns. We will prove that this permutation either meets condition 1 or condition 2. Assume as the inductive hypothesis that this claim is true if $w$ were to have $m$ or less $321$-leading segments (with a base case of $m=0$). Assume $w$ has $m+1$ $321$-leading segments. The $321$-leading factor of $w$ is alternating, has peaks decreasing from left to right, and dips increasing from left to right. The peaks of the $321$-leading factor of $w$ must be the largest letters in the $w$ because otherwise a $123$ or $213$ would occur somewhere. By the inductive hypothesis, the portion of $w$ not included in its $321$-leading factor satisfies condition 1. Finally, if there are more than one $321$-leading segments in $w$, then where $k$ is the $k$-length of $w$, $k'$ is the number of peaks in the second to left-most $321$-leading segment in $w$, and $j$ is the value of the final dip in the $321$-leading factor of $x$, $k+k' \le n-j$. This inequality must hold because otherwise a peak in the second $321$-leading segment in $w$ which has value less than $j$ could be moved to the position two to the right of $j$, forming a $231$ pattern, a contradiction. Thus $w$ is compact.
\end{proof}

\begin{prop}
Let $g(n, k)$ count the number of permutations meeting condition $1$ for compactness and with a $321$-leading factor containing $k-1$ dips and $k$ peaks. We assume that $k$ is such that at least one such permutation exists. Then,
$$
 g(n, k) =
\begin{cases}
 1 , & \text{if } n=1 \text{ or } n-2k+1=0  \\
 \sum\limits_{j=1}^{\lfloor (n-1)/2 \rfloor}{g(n-1, j)} , & \text{if } k=1  \\
 \sum\limits_{x=k-1}^{n-k}{\sum\limits_{j=1}^{n-k-x}{\binom{x-1}{k-2} \cdot g(n-2k+1, j)}},  & \text{ otherwise.}
\end{cases}
$$
\end{prop}
\begin{proof}
If $n=1$, then it is trivial that $k=1$ and $g(n,k)=1$. If $n-2k+1=0$, then there is only one permutation which $g(n,k)$ counts, having peaks in decreasing order and containing the larger half of the letters, and the dips increasing from left to right and containing the remaining letters.

Otherwise, if $k=1$, then we need the number of possible permutations of size $n-1$ satisfying condition 1. (We know $n-1>0$; we can simply append $n$ to each of these permutations.) We define $j$ as the number of peaks in the $321$-leading factor of a permutation in $S_{n-1}$. Then, $j$ can be anywhere from $1$ to $\lfloor (n-1)/2 \rfloor$. So, $$g(n,k)=\sum\limits_{j=1}^{\lfloor (n-1)/2 \rfloor}{g(n-1, j)}\text{.}$$

Otherwise, we know that $n>1$, $k>1$, and $n-2k-1 \neq 0$. Let $w$ be a permutation counted by $g(n,k)$ in this case. We define $x$ to be the value of the final dip in the $321$-leading factor of $w$. Because the dips in the $321$-leading factor increase from left to right, $x$ is at least $k-1$. Because the largest $k$ letters in $w$ are used in peaks of the $321$-leading factor, $x$ is at most $n-k$. However, $x$ can be any value in-between inclusive. We define $j$ to be the number of peaks in the second $321$-leading segment of $w$. Note that $j$ can be anywhere between $1$ and $n-k-x$ inclusive (by the inequality in condition 1). Now, given $x$ and $j$, we choose $k-2$ dips in the $321$-leading segment of $w$ ($x$ is already chosen). They can have any values less than $x$. So, there are $\binom{x-1}{k-2}$ choices. Then, we choose $w'$, the permutation created by the letters to the right of the $321$-leading factor of $w$. There are $g(n-2k+1, j)$ choices for $w'$. So, we get
$$g(n,k)=\sum\limits_{x=k-1}^{n-k}{\sum\limits_{j=1}^{n-k-x}{\binom{x-1}{k-2} \cdot g(n-2k+1, j)}}\text{.}$$ 
\end{proof}

\begin{prop}
Let the function $g$ be defined as in the previous proposition. Let $f(n)$ be the number of classes in $S_n$ under the $\{123, 231\} \{213, 312\}$-equivalence. Then, 
$$f(n)=\sum\limits_{k=1}^{\lfloor n/2+1 \rfloor}{g(n+1, k)}+n-2\text{.}$$
\end{prop}
\begin{proof}
First, we calculate the number of classes containing only $123$ and $231$ avoiding permutations. The number of these containing permutations starting with a decrease is $$\sum\limits_{k=1}^{\lfloor (n-1)/2 \rfloor}{g(n+1, k)}\text{.}$$ This is also the same as the number of such classes in $S_{n+1}$ containing only permutations starting with a $321$ pattern (we create a bijection by just appending $n+1$ to the left). The number of such classes $S_n$ that contain permutations starting with an increase is the same as the number of such classes in $S_{n+1}$ containing permutations starting with a decrease but not with a $321$. This falls from the definition of condition $2$. Adding these together, we get that the number of such classes in $S_n$ is the number of such classes in $S_{n+1}$ starting with a $321$ added to the number of such classes in $S_{n+1}$ starting with a decrease but not a $321$, which is just the number of classes containing only $123$ and $231$ avoiding permutations starting with a decrease in $S_{n+1}$. We know that this is $\sum\limits_{k=1}^{\lfloor n/2+1 \rfloor}{g(n+1, k)}$.

Now, will calculate the number of classes containing at least one permutation which contains a $123$ or $231$ pattern. First we note an invariant. Let $w$ be a permutation in $S_n$. If $1$ and $2$ are of the same position parity as each other and $1$ is to the left of $2$, then they cannot be involved in a hit with each other until one of them changes position parity under the transformations considered. However, unless they are in the same hit, they always both act as $1$ in any hit, thus maintaining their position parity. So, in this case, $1$ and $2$ can never swap relative order and can never change position parity under the transformations considered. Then, if $3$ is to the right of $2$ and of the same position parity, then $3$ and $2$ can never swap relative order or change position parity. Similarly, if the $k$ smallest letters are ordered increasing from left to right and are of the same position parity, then none of their position parities can ever change and none of them can ever swap relative order. So, the highest $k$ for which this is true and the parity of the lowest $k$ letters for a permutation does not change under the transformations considered. Let $k$ be that highest such $k$ for a given permutation $w$ in $S_n$. If $1$ has an odd position, $k$ can have any of $\lceil n/2 \rceil$ values. If $1$ has an even position, $k$ can have any of $\lfloor n/2 \rfloor$ values. However, in the case where every odd position is less than every even position or vice versa, $w$ is in one of the $123$ and $231$ avoiding classes. Hence, there are at least $n-2$ classes which contain a permutation containing $123$ or $231$ patterns fall into.
 
We will now show that permutations containing a $123$ or $231$ fall into exactly $n-2$ classes. Let $w \in S_n$ contain either a $123$ or $231$ pattern. Let $k$ be the largest $k$ such that the lowest $k$ letters are in increasing order from left to right and each of the same position parity in $w$. Let $p$ be the position parity of $1$ in $w$. We will show that $w$ is equivalent to $w'$, the permutation which increases from left to right except with $1$ inserted in the left-most position with parity $p$ and the next $k-1$ letters in value inserted every two positions to its right. (e.g., if $k=3$, $n=8$, and $1$ is of odd position parity, then $w'=14253678$.) Assume as an inductive hypothesis that this claim holds in $S_{n-1}$ (with the base cases of $n \le 5$ easy to check). Note that all such $w'$ have their final three letters in increasing order because we are not considering the cases where all the letters of one position parity are greater than those of the other.

If a $123$ or $231$ pattern occurs only in the final three letters of $w$, then applying the inductive hypothesis to the final $n-1$ letters, we reach a permutation where that is not the case. Hence, we may assume that $w$ has a $123$ or $231$ pattern in its first $n-1$ letters. We apply the inductive hypothesis to them, reaching a permutation $x$ which has the three letters preceding its final letter in increasing order. If $1$ is in the first position of $x$, then we apply the inductive hypothesis to the right-most $n-1$ letters of $x$ and reach $w'$. If $1$ is in the second position in $x$, then we apply the inductive hypothesis to the right-most $n-1$ letters of $x$ to reach $x'$ which has its final three letters in increasing order. If $x'$ is not $w'$, then through an application of $312 \rightarrow 213$, we place the first letter in $w'$ in the first position in $x'$ (note that this letter was previously in the third position in $x'$). Then, applying the inductive hypothesis to the right-most $n-1$ letters, we reach $w'$. If $1$ is in the final position of $x$, then we apply the inductive hypothesis to the final $n-1$ letters, reaching a permutation which starts with some letter $t$ and increases afterwards, $m$. Applying the inductive hypothesis to the first $n-1$ letters of $m$ and then the to the final $n-1$ letters, we are done (and we reach the identity). Hence, $w$ and $w'$ are always equivalent and the proposition holds.
\end{proof}

\subsection{\mathheader{\{123, 132\} \{213, 231 \}}-Equivalence}

\begin{prop}\label{prop2n} There are $2^{n-1}$ classes in $S_n$ under the \\$\{123, 132\} \{213, 231 \}$-equivalence.
\end{prop}
\begin{proof}
For $n\le 5$, this can be shown computationally. For $n>5$, this falls from Theorem \ref{thmavoiders} (proved in Subsection \ref{Secconnect}).
\end{proof}

\begin{cor}\label{corollarysamesizes}
Let $w$ be in $S_{n>1}$ and $w'$ be $w$ with $n$ struck. The size of the equivalence class of $w$ under the  $\{123, 132\} \{213, 231 \}$-equivalence is the size of the class containing $w'$ if $n$ is in the first position in $w$ and is the size times $n-1$ of the class containing $w'$ otherwise.
\end{cor}
\begin{proof}
For the sake of brevity, we leave this as a simple exercise for the reader.
\end{proof}

\begin{cor}
The number of permutations in $S_n$ in the class containing the identity under the $\{123, 132\} \{213, 231 \}$-equivalence is $(n-1)!$. 
\end{cor}
\begin{proof}
As a result of Corollary \ref{corollarysamesizes}, this falls from straightforward computation.
\end{proof}

\subsection{\mathheader{\{123, 321\} \{ 213, 231\}}-Equivalence}
In this subsection, we use several results proven in Subsection \ref{secstooge}.

\begin{lem}\label{bigclasslem}
All non-avoiding permutations are equivalent in $S_n$ for $n>5$ under the $\{123, 321\} \{ 213, 231\}$-equivalence.
\end{lem}
\begin{proof}
Assume inductively that the result holds for smaller $n$ (with a base case of $n=6$). Let $n>6$. By Proposition \ref{propreachmiddled}, each non-avoiding permutation in $S_n$ is equivalent to some middled permutation. It follows from Theorem \ref{thmstoogesort} that every permutation is equivalent to the identity in $S_n$.
\end{proof}

\begin{prop}
There are $3$ classes in $S_n$ under the $\{123, 321\} \{213, 231\}$-equivalence for $n>5$.
\end{prop}
\begin{proof}
By Lemma \ref{bigclasslem}, there is one nontrivial class. It is easy to check that there are always two trivial classes.
\end{proof}

\begin{cor}
The class containing the identity in $S_n$ contains $n!-2$ elements for $n>5$ under the $\{123, 321\} \{ 213, 231\}$-equivalence.
\end{cor}
\begin{proof}
This is because all $n!-2$ of the non-avoiding permutations are reachable from the identity.
\end{proof}

\subsection{\mathheader{\{123, 231\} \{132, 321\}}-Equivalence}

\begin{prop} There are $2^{n-1}$ classes in $S_n$ under the \\$\{123, 231\} \{132, 321\}$-equivalence.
\end{prop}
\begin{proof}
One wants to show inductively that the permutations with the letter $1$ in an odd position (type 1) break into $2^{n-2}$ classes and that those with $1$ in an even position (type 2) break into $2^{n-2}$ classes, thus completing the proof. For the sake of brevity, we leave this as an exercise for the reader. 


\end{proof}

\begin{defn}
A permutation is \emph{trivializable} if is equivalent to the identity.
\end{defn}

\begin{prop} In $S_n$, under the $\{123, 231\} \{132, 321\}$-equivalence, the number of trivializable permutations is $\lceil n/2 \rceil ! ^2$ if $n$ is even and $\lceil n/2 \rceil ! ^2 / \lceil n/2 \rceil$ if $n$ is odd. By symmetry, the same holds for the number of permutations equivalent to the descending permutation.
\end{prop} 
\begin{proof}
We leave this proof as an exercise for the reader. It basically relies on the fact that the permutations in the class containing the identity in $S_n$ are the permutations in the class containing the identity in $S_{n-1}$, except with a new smallest letter inserted in an odd-numbered position.
\end{proof}

\subsection{\mathheader{ \{ 132, 231 \} \{213, 312\}}-Equivalence Experimental Data}

Figure~\ref{openprobtable} provides computational data for the number of classes created in $S_n$ by the $\{132, 231 \} \{213, 312 \}$-equivalence for $n \le 12$. This equivalence is the only replacement partition of $S_3$ with two nontrivial parts, each of size two, which has yet to be enumerated.

\begin{figure}[h] 
\caption{The number of classes created in $S_n$ by the $\left\{132, 231\right\}\left\{ 213, 312\right\}$-equivalence.}
\begin{center}
    \begin{tabular}{ | l | | l  | l  | l  | l  | l | l  | l  | l  | l  | l |}
    \hline
     $n$       & 3 & 4 &  5&  6&   7&   8&    9&    10&    11& 12 \\ \hline
     \# classes & 4 & 10& 26& 76& 234& 782& 2804& 10972& 47246&  224648 \\ \hline 
     \end{tabular}
\end{center}
\label{openprobtable}
\end{figure}

\section{General Results}\label{Secgen}

In this section, we discuss several general trends that arise in the study of pattern-replacement relations, such as when equivalence classes can be counted using pattern-avoidance (Subsection \ref{Secconnect}), when two pattern-replacement relations are the same (Subsection \ref{secsame}), and the role of stooge-sort-like algorithms in counting equivalence classes (Subsection \ref{secstooge}).

\subsection{Connecting pattern-avoidance to the enumeration of equivalence classes}\label{Secconnect}

In this subsection, we study the connection between pattern-avoidance and the enumeration of equivalence classes.

Let $P$ be a partition of $S_c$. Let $D$ be the set of permutations which are lexicographically smaller than every other permutation in their part. Let $U=S_c \setminus D$. 

We define $N_n$ as the number of classes in $S_n$ under the $P$-equivalence. We define $A_n$ as the number of permutations in $S_n$ which avoid each permutation in $U$.

In Theorem \ref{thmavoiders}, we show that for any $k \ge 2c-1$, $N_k = A_k \implies N_n=A_n$ for all $n\ge k$. Although the result is quite simple, its effects are wide-ranging. Among its uses in this paper have been to prove that there are $2^{n-1}$ classes in $S_n$ under the $\{123, 132\}\{213, 231\}$-equivalence (Proposition \ref{prop2n}), and as one proof that there are $2^{n-1}$ classes in $S_n$ under the $\{123, 132, 231\}$-equivalence (Remark \ref{remshortproof}). The fact that we can use it in the case of the $\{123, 132, 231\}$-equivalence in particular says something interesting about the power of Theorem \ref{thmavoiders}. Indeed, the other proof that we provide for the enumeration uses a clever complete invariant; it seems very unintuitive that we can actually prove the enumeration without ever noting the very invariant that seems to causes it. \cite{LPRW10} studied pattern-replacement equivalences where one does not require that letters in patterns be consecutive in a permutation; this modified concept of an equivalence was referred to as $P^{\fourdots}$-equivalence for the replacement partition $P$. It is trivial to modify the proof of Theorem \ref{thmavoiders} so that it holds for these equivalence relations as well\footnote{For the sake of brevity, we do not provide the modified proof here. However, one only needs to redefine $N_n$ as the number of classes in $S_n$ under the $P^{\fourdots}$-equivalence, redefine the notion of avoidance to not require that letters in patterns be adjacent, and to redefine $f$ in the proof of Theorem \ref{thmavoiders} as a subword rather than a factor.}. Thus, the result can also be used to provide very short alternative proofs of the enumeration of the classes for each of the $\{123, 132\}^{\fourdots}$-equivalence and the $\{123, 132, 213\}^{\fourdots}$-equivalence. Although the theorem serves as an alternative proof for each of these results, one should not consider it to substitute their proofs; each of the already known proofs provide insight into the characterization of the equivalence classes in $S_n$, while the theorem does little more than to enumerate them.

\begin{defn}
We say that a rearrangement of a hit from $U$ to be a hit from $D$ is a \emph{down jump.}
\end{defn}

\begin{defn}
We call $w$ an \emph{avoider} if $w$ avoids each permutation in $U$.
\end{defn}

Note that this notion of an avoider is slightly different from the one used in the other sections of the paper.

\begin{defn}
We say that the \emph{height} of a permutation $w$ is the largest $k$ such that through $k$ down jumps, we may go from $w$ to an avoider (Lemma \ref{lemreduction} shows that such a $k$ exists).
\end{defn}

\begin{defn}
We say that a permutation is \emph{secure} if after repeatedly performing down jumps to it, we always reach the same avoider.
\end{defn}
\begin{defn}
A pair of permutations is \emph{matched} if each permutation in the pair is secure, and repeated down jumps applied to one of the permutations in the pair brings us to the same avoider as repeated down jumps applied to the other.
\end{defn}

\begin{lem}\label{lemreduction}
Let $w \in S_n$. If we repeatedly perform down jumps on $w$, we will eventually reach an avoider.
\end{lem} 
\begin{proof}
Each down jump brings us to a permutation which is lexicographically smaller. Since there are a finite number of permutations in $S_n$, repeated down jumps must always eventually bring us to an avoider.
\end{proof}

\begin{thm}\label{thmavoiders}
Let $k \ge 2c-1$. If $N_k = A_k$, then $N_n=A_n$ for all $n\ge k$.
\end{thm}
\begin{proof}
Since $N_k=A_k$, we may conclude that each class in $S_k$ contains exactly one avoider (by Lemma \ref{lemreduction}). Hence, each permutation in $S_k$ is secure.

Assume that each of the permutations in $S_{n \ge k}$ of height $\leq h-1$ are secure. We will show that so are each of the permutations in $S_n$ of height $h$. The base case for this induction of $h=0$ is trivial.

Let $w$ be a permutation in $S_n$ of height $h$. Let $a$ and $b$ be permutations which can be reached from $w$ through a down jump using hits $x$ and $y$ respectively. By the inductive hypothesis, $a$ and $b$ are secure. By showing that $a$ and $b$ are matched, we will establish that $w$ is secure. 

If $x$ and $y$ are the same, then $a=b$ and $a$ and $b$ are trivially matched.

If $x$ and $y$ are disjoint, then let $w'$ be $w$ except with each hit $x$ and $y$ rearranged in the same manner as in $a$ and $b$ respectively. Then, $w'$ is matched to each of $a$ and $b$ since it can be reached from either by a down jump and because of the inductive hypothesis. Hence, $a$ and $b$ are matched.

If $x$ and $y$ overlap but are not equal, then we consider a factor\footnote{As noted previously, to modify the proof to hold for the $P^{\fourdots}$-equivalence, we simply consider a subword rather than a factor.} $f$ of $w$ containing both $x$ and $y$ which is of size $k$. We define $f'$ as the avoider reached from $f$ through repeated down jumps. ($f$ is secure since it is in $S_k$.) Rearranging $f$ as $f'$ in $w$, we reach a secure permutation $w'$ by the inductive hypothesis. Furthermore, since $f$ is secure, $f'$ and $f$ with either of $x$ or $y$ rearranged in a down jump are matched. This implies that $a$ is matched to $w'$ and $b$ is matched to $w'$. Hence, $a$ and $b$ are matched.

Thus, each permutation in $S_{n>k}$ is secure. Hence, no two avoiders in $S_n$ are equivalent. Since every permutation in $S_n$ is equivalent to an avoider (Lemma \ref{lemreduction}), $N_n=A_n$.
\end{proof}

\subsection{Two Relations Can Be Equivalent}\label{secsame}

It is easy to not notice the importance of a very simple equivalence between pattern-replacement relations. For a partition $P$ of $S_c$, in $S_{n>c}$, the $P$-equivalence is the same as the $P'$-equivalence where $P'$ is the partition of $S_{c+1}$ into which the $P$-equivalence partitions $S_{c+1}$. For the $\{213, 231, 312, 321\}$-equivalence (enumerated by Proposition \ref{propsimplen}), for example, this is fairly useful. It tells us that our results describing the equivalence classes under the relation also describe the classes under the $\{W \subset S_c | W=\{ w\in S_c | w_k=n-k+1 \text{ for } k<i, w_i\neq n-i+1 \text{ if } i \neq n \}, 1 \le i \le n \}$-equivalence in $S_{n>c}$ for any $c$. This equivalence between pattern-replacement relations also plays an interesting role with regard to Theorem \ref{thmavoiders}. In fact, a pattern-replacement relation that is not confluent when expressed with a replacement partition of $S_c$ can be confluent when expressed as a replacement partition of $S_{k>c}$.

We now consider when the $P$-equivalence and $P^{\fourdots}$-equivalence are the same.

\begin{thm}\label{Thmsame2}
Let $P$ be a partition of $S_c$ such that the equivalence classes in $S_k$ are the same under both the $P$-equivalence and the $P^{\fourdots}$-equivalence for some $k>c$. Then, the equivalence classes in $S_n$ are the same under both the $P$-equivalence and the $P^{\fourdots}$-equivalence for all $n\ge k$
\end{thm}
\begin{proof}
Assume that the theorem holds in $S_{n-1}$ with an inductive base case of $n=k$. Let $w$ and $w'$ in $S_n$ be two permutations equivalent under the $P^{\fourdots}$-equivalence through a single replacement using the hit $h$ in $w$. By showing that $w$ and $w'$ are also equivalent under the $P$-equivalence, we will have completed the proof (since the other direction of implication is trivial).

Let $j$ be the left-most position that does not contain a letter in $h$ in $w$. Let $x$ be $w$ with $w_j$ struck and $x'$ be $w'$ with $w_j$ struck. By the inductive hypothesis, there is a series of replacements (using hits with adjacent letters only) bringing us from $x$ to $x'$. We may construct a series of transformations bringing us from $w$ to $w'$ using exactly the same replacements. In doing this, we use only hits each of which fit into a factor of size $c+1$. By the inductive hypothesis, each such transformation is a valid rearrangement under the $P$-equivalence and hence $w \equiv w'$ under the $P$-equivalence. 
\end{proof}

An example relation to which Theorem \ref{Thmsame2} is applicable is the $\{123, 132, 213, 231\}$-equivalence. (See Subsection \ref{sub4s}.) 

\subsection{Generalizing Stooge Sort}\label{secstooge}

The following results help to encapsulate the idea of a stooge sort\footnote{Stooge sort is a sorting algorithm in which one recursively sorts the first two thirds, than the second two thirds, and then the first two thirds again of a list. It is easy to check that this results in a completely sorted list. } and its role in pattern-replacement relations. Previously in this paper, they have played crucial roles in enumerating classes under the $\{123, 132, 231, 321\}$-equivalence, the $\{213, 132, 231, 312\}$-equivalence, the $\{123, 321\}\{213, 231\}$-equivalence, and most significantly, the $\{123, 132, 321\}$-equiva\-lence (the enumeration of which was a previously open problem). 

Consider the $P$-equivalence where $P$ partitions $S_c$. 

In this subsection, we assume that one has a pre-picked relation between permutations in which they are well-ordered such that rearranging a factor of a permutation to be smaller also brings us to a smaller permutation (like lexicographic comparison). Each time that the results are used in this paper, it is implicitly assumed we are using lexicographic comparison.

\begin{defn}
A permutation is \emph{lefted} if it contains a hit in its final $n-1$ letters. 
\end{defn}
\begin{defn}
A permutation is \emph{righted} if it contains a hit in its first $n-1$ letters.
\end{defn}
\begin{defn}
A permutation is \emph{middled} if it contains a hit which involves neither the final or first letter of the permutation.
\end{defn}
\begin{defn}
Let $n \geq c+1$. We define $L_n$ as the set containing the each lefted permutation which is the smallest lefted permutation in its equivalence class in $S_n$. 
\end{defn}
\begin{defn}
Let $n \geq c+1$. We define $R_n$ as the set containing the each righted permutation which is the smallest righted permutation in its equivalence class in $S_n$.
\end{defn}
\begin{defn}
Let $n \geq c+2$. We define $I_n$ as the set containing each middled permutation which is the smallest middled permutation in its equivalence class in $S_n$.
\end{defn}

\begin{defn}
Let $w \in S_n$ be middled. Then, we define $l(w)$ as $w$ with its left-most $n-1$ letters rearranged to form the element of $L_{n-1}$ that they are equivalent to. (Note that such an element \emph{does} exist.) Similarly, we define $r(w)$ as $w$ with its right-most $n-1$ letters rearranged to form the element of $R_{n-1}$ that they are equivalent to. Both $r(w)$ and $l(w)$ are middled.
\end{defn}

\begin{prop}\label{propreachmiddled}
If each non-avoiding permutation in $S_{n-1}$ is equivalent to some middled permutation, then the same is true in $S_n$.
\end{prop}
\begin{proof}
Let $w \in S_n$ be non-avoiding. By the assumption, we can rearrange one of the first $n-1$ letters of the last $n-1$ letters to form a middled permutation in $S_{n-1}$, bringing us to a middled permutation in $S_n$.
\end{proof}

\begin{thm}\label{thmstoogesort}
Let $w \in S_n$ for $n\ge c+2$ be middled. Then, $w$ is equivalent to a middled permutation $w'$ such that the first $n-1$ letters of $w'$ form an element of $L_{n-1}$ and the final $n-1$ letters of $w'$ form an element of $R_{n-1}$.
\end{thm}
\begin{proof}
Unless the first $n-1$ letters of $w$ form an element of $L_{n-1}$, then $l(w)<w$. Unless the final $n-1$ letters of $w$ form an element of $R_{n-1}$, then $r(w)<w$. But since there are a finite number of permutations in $S_n$, we can not keep reaching smaller permutations over and over again in an unending process. Instead, through repeated applications of $r$ and $l$ to $w$, we must reach a permutation whose first $n-1$ letters form an element of $L_{n-1}$ and final $n-1$ letters form an element of $R_{n-1}$.
\end{proof}

\begin{defn}
We say a hit is a \emph{global minimum} if it is the smallest hit in some nontrivial part of $P$.
\end{defn}

Now we study the connections between $L_n$, $R_n$, and $I_n$.

\begin{prop}
Assume that we are using lexicographic comparison to compare permutations. Let $m\in I_n$. Then, the element $w$ of $L_n$ such that $w\equiv m$ exists and either is $m$ or is $m$ with its final $c$ letters rearranged to form a global minimum.
\end{prop}
\begin{proof}
Let $m \in I_n$ and $w \in L_n$ be such that $w \equiv m$. (It is trivial that $w$ exists.) Let $m'$ and $w'$ be the factors of $m$ and $w$ containing the first $n-c$ letters of each. Assume $w \neq m$ and thus $w < m$. If $w$ contains a hit in the first $n-1$ letters, then we have a contradiction since $m \in I_n$, not $w \in I_n$. Since $w \equiv m$, some rearrangement of the hit in $w$ must create a hit which is in the first $n-1$ letters. Hence there exists a middled permutation $x \equiv w$ such that the first $n-c$ letters of $x$ are the same as those of $w$. Hence if $w' \neq m'$ and thus $w' < m'$, then we have a contradiction since this would imply that $x \in I_n$, not $m \in I_n$. Therefore, and $w$ is simply $m$ with its final $c$ letters rearranged to form a global minimum.
\end{proof}

Note that there is not a similar proposition to be said for $L_n$. In fact, the closest we can come is with the following proposition.

\begin{prop}
\textbf{(a)} Assume that no elements of $R_{n-1}$ contain a hit in the first $c$ letters. Then the same is true for $R_n$. \textbf{(b)} Assume that no elements of $L_{n-1}$ contain a hit in the final $c$ letters. Then, the same is true for $L_n$.
\end{prop}
\begin{proof}
Assume that for our method of comparison of permutations, we have that rearranging a factor of a permutation to be smaller always yields a smaller permutation. Assume that no elements of $R_{n-1}$ contain a hit in the first $c$ letters. Assume that there is an element $w$ of $R_n$ containing a hit in the first $c$ letters. Then, rearranging the first $n-1$ letters of $w$ to form a permutation in $R_{n-1}$ (which we can do since the first $n-1$ letters of $w$ form a righted permutation), we reach a smaller righted permutation in $S_n$ than $w$, a contradiction. This proves \textbf{(a)}; \textbf{(b)} can be proved analogously.
\end{proof}

\section{Conclusion and Future Work}\label{conclusion}
In our study of pattern-replacement relations, several new recurring ideas came to light. First, in proofs by induction, we found that modified versions of stooge sort could often be used. Second, we found that sometimes it is easier to treat nontrivial and trivial classes separately. In addition, several previously known tools played an important role in our research, for example the search for invariants and systems of representatives (sets of permutations such that each class has exactly one element of the set).

While the proofs of our results use a common set of tools, they are creatures of different ilks: they vary in difficulty and structure. This makes it all the more surprising that the results show some unexpected similarities, like the $2^{n-1}$ occurring four times in our enumerations of the classes in $S_n$ (and the multiset of class sizes in $S_n$ being the same in two cases). If our results are combined with those of previous works, \cite{PRW11} and \cite{LPRW10}, even more similarities occur. Analyzing the sources of these similarities, as well as viewing our equivalence relations from the more advanced viewpoint of algebra, is a promising direction for future study.  It would be interesting to find formulas for the number of classes created by several replacement partitions not yet well-understood, specifically $\{132, 231\} \{213, 312\}$. In addition, there are the following directions of future work:
\begin{enumerate}
\item In Section \ref{Secdouble}, we only provide the size of the class containing the identity for relations when the result is convenient. Future authors might further study the sizes of the classes created under those relations.
\item Are there connections between equivalence relations having the same number of classes. Is there a reason why the enumerations for each of the Knuth relation and the forgotten relation show up again in our study of the $\{123, 132\}\{213, 312\}$-equivalence and $\{123, 231\} \{213, 132\}$-equivalence respectively?
\item \cite{LPRW10} deals with relations that allow re-ordering only adjacently valued letters or, alternatively, re-ordering any subword (rather than only a contiguous set of adjacent letters). This is an important direction of research to continue.
\item Some equivalence classes have additional structure. Can one classify permutations in a given equivalence class based on characteristics such as inversions, length (number of inversions), the locations of hits, ascents, Major index, etc.?
\item We can consider the $K$-equivalence not only on permutations, but on arbitrary words (cf.\ the Knuth relations). By linearization, this corresponds to studying binomial ideals in noncommutative polynomial rings, and the quotient rings modulo these ideals. These have revealed interesting properties in the cases of the Knuth and forgotten relations. Do similar properties arise in the relations that we have studied?
\end{enumerate}

\subsection*{Acknowledgements}

I would like to thank Richard Stanley as well as the MIT PRIMES program for providing me with this research project. I would like to thank Sergei Bernstein, Darij Grinberg, and Ziling Zhou for many useful conversations throughout the research, as well as for helping with the editing process of this paper. I would also like to thank Tanya Khovanova for suggesting edits to this paper.


\end{document}